\input ./style/arxiv-general.cfg
\documentclass[aop,MSNbibl,seceqn,dvips]{arximspdf}
\makeatletter
   \@ifpackageloaded{graphicx}{}{\usepackage{graphicx}}
\makeatother

\doi{10.1214/15-AOP1014}
\volume{44}
\issue{3}
\pubyear{2016}
\firstpage{1956}
\lastpage{1984}
\docsubty{FLA}

\makeatletter
\newcommand{\rrvert}{\vert}
\newcommand{\llvert}{\vert}
\newcommand{\iint}{\int\!\!\!\int}
\newcommand{\eqref}[1]{(\ref{#1})}
\newcommand{\prob}{\P}
\newcommand{\diam}{\operatorname{diam}}
\newcommand{\carr}{\operatorname{carr}}
\newcommand{\supp}{\operatorname{supp}}
\newcommand{\cone}{\mathrm{Cone}}
\def\Veuc{V_{\mathrm{euc}}}
\def\Beuc{B_{\mathrm{euc}}}
\def\sE{{\mathcal E}}
\def\sD{{\mathcal D}}
\def\I{{\mathcal{I}}}
\newcommand{\eps}{\varepsilon}
\def\Reff{R_{\mathrm{eff}}}
\def\oX{{\overline X}}
\newcommand{\cH}{{\mathcal H}}
\newcommand{\cM}{{\mathcal M}}
\newcommand{\leb}{\operatorname{\mathfrak{Leb}}}
\renewcommand{\P}{\mathbb P}
\newcommand{\Z}{\mathbb Z}
\newcommand{\E}{\mathbb E}
\newcommand{\R}{\mathbb R}

\newtheorem{thmm}{Theorem}[section]
\newtheorem{theorem}[thmm]{Theorem}
\newtheorem{lemma}[thmm]{Lemma}
\newtheorem{corollary}[thmm]{Corollary}

\newtheorem{proposition}[thmm]{Proposition}
\newproclaim{defn}[thmm]{Definition}
\newproclaim{remark}[thmm]{Remark}
\newproclaim{example}[thmm]{Example}

\makeatother

\begin{document}
\begin{frontmatter}

\title{Boundaries of planar graphs, via circle packings\thanksref{TT1}}
\runtitle{Boundaries of planar graphs, via circle packings}
\thankstext{TT1}{Supported NSERC grants.}

\begin{aug}
\author[A]{\fnms{Omer}~\snm{Angel}\thanksref{m1,TT2}\ead[label=e1]{angel@math.ubc.ca}},
\author[A]{\fnms{Martin T.}~\snm{Barlow}\thanksref{m1}\ead[label=e2]{barlow@math.ubc.ca}},
\author[B]{\fnms{Ori }~\snm{Gurel-Gurevich}\thanksref{m2}\ead[label=e3]{origurel@math.huji.ac.il}}
\and
\author[C]{\fnms{Asaf}~\snm{Nachmias}\corref{}\thanksref{m1,m3}\ead[label=e4]{asafnach@math.ubc.ca}}
\runauthor{Angel, Barlow, Gurel-Gurevich and Nachmias}
\affiliation{University of British Columbia\thanksmark{m1}, Hebrew
University\thanksmark{m2} and Tel Aviv University\thanksmark{m3}}
\thankstext{TT2}{Supported in part by IHES, OGG by
PIMS, and AN by NSF grant.}
\address[A]{O. Angel\\
M. Barlow\\
Department of Mathematics\\
University of British Columbia\\
Vancouver, British Columbia V6T 1Z2\\
Canada\\
\printead{e1}\\
\phantom{E-mail:\ }\printead*{e2}}
\address[B]{O. Gurel-Gurevich\\
Einstein Institute of Mathematics\\
Hebrew University of Jerusalem\\
Givat Ram, Jerusalem, 91904\\
Israel\\
\printead{e3}}
\address[C]{A. Nachmias\\
Department of Mathematics\\
University of British Columbia\\
Vancouver, British Columbia V6T 1Z2\\
Canada\\
and\\
Department of Mathematical Sciences\\
Tel Aviv University\\
Tel Aviv 69978\\
Israel\\
\printead{e4}}
\end{aug}

%
\received{\smonth{11} \syear{2013}}
%
\revised{\smonth{1} \syear{2015}}

%
\begin{abstract}
We provide a geometric representation of the Poisson and Martin
boundaries of a transient, bounded degree triangulation of the plane in
terms of its circle packing in the unit disc. (This packing is unique up
to M\"obius transformations.) More precisely, we show that any bounded
harmonic function on the graph is the harmonic extension of some
measurable function on the boundary of the disk, and that the space of
extremal positive harmonic functions, that is, the Martin boundary, is
homeomorphic to the unit circle.

All our results hold more generally for any ``good''-embedding of planar
graphs, that is, an embedding in the unit disc with straight lines such
that angles are bounded away from $0$ and $\pi$ uniformly, and lengths of
adjacent edges are comparable. Furthermore, we show that in a good
embedding of a planar graph the probability that a random walk exits a
disc through a sufficiently wide arc is at least a constant, and that
Brownian motion on such graphs takes time of order $r^2$ to exit a disc
of radius $r$. These answer a question recently posed by Chelkak
(2014).
\end{abstract}

%
\end{frontmatter}

\section{Introduction}\label{sec1}

Given a Markov chain, it is natural to ask what is its ``final'' behavior,
that is, the behavior as the time tends to infinity. For example, consider
the lazy simple random walk on a rooted $3$-regular tree---the path of
the random walk almost surely determines a unique infinite branch of the
tree. This branch is determined by the tail $\sigma$-field of the random
walk and moreover, this $\sigma$-field is characterized by the set of such
infinite branches. In general, it is more useful to consider the \emph{invariant} $\sigma$-field $\I$, that is, all the events that are
invariant under the time-shift operator. In the case of lazy Markov chains,
these two $\sigma$-fields are equivalent \cite{De,Ka}.

To any invariant event $A$ we can associate a harmonic function $h_A$ on
the state space by $h_A(x) = \prob_x(A)$, that is, the probability
that $A$
occurs starting the chain from $x$. (A function $h$ is harmonic if its
value at a state is the expected value of $h$ after one step of the chain.)
In fact, there is a correspondence between bounded invariant random
variables and bounded harmonic functions on the state space (given such a
random variable $Y$, the function is $h_Y(x) = \E_x(Y)$, see
\cite{Ka,LP}). Thus, the set of bounded harmonic functions on the state
space characterizes all the ``final'' behaviors of the Markov chain.

In this paper, we consider reversible Markov chains in discrete time and
space (i.e., weighted random walks on a graphs). It is not hard to see
that if the chain is recurrent, then there are no nonconstant bounded
harmonic functions. On the other hand, transience does not guarantee the
existence of such functions, as can be seen in the simple random walk on
$\Z^3$. However, in the planar case there is such a dichotomy: Benjamini
and Schramm \cite{BS2} proved that if $G$ is a transient, bounded degree
planar graph, then $G$ exhibits nonconstant bounded harmonic functions.

The proof in \cite{BS2} relies on the theory of circle packing. Recall that
a \emph{circle packing} $P$ of a planar graph $G$ is a set of circles
with disjoint interiors $\{C_v\}_{v\in G}$ such that two circles are
tangent if and only if the corresponding vertices form an edge. Koebe's
circle packing theorem \cite{Ko} states that any planar graph has a circle
packing, and that for triangulations (graphs where all faces are triangles)
the circle packing is essentially unique. Given a circle packing, we embed
the graph in $\R^2$, with straight line segments between the corresponding
centers of circles for edges. The \emph{carrier} of $P$, denoted
$\carr(P)$, is the union of all the closed polygons corresponding to
the faces. He
and Schramm \cite{HS} provided an insightful connection between the
probabilistic notion of recurrence or transience of $G$ and the
geometry of
$\carr(P)$. Their theorem states that if $G$ is a bounded degree one-ended
triangulation, then it can be circle packed so that the carrier is either
the entire plane or the open unit disc $U$ according to whether $G$ is recurrent
or transient, respectively. Since we are interested in nonconstant bounded
harmonic functions, we consider here only the latter case.

Consider a transient, bounded degree, one-ended triangulation $G$ and its
circle packing $P=\{C_v\}_{v\in V}$ with $\carr(P)=U$.
We identify each
vertex $v$ with the center of $C_v$---it will always be clear from the
context if the letter $v$ represents a vertex or a point in $\R^2$. Let
$\{X_n\}$ be the simple random walk on $G$. A principal result of Benjamini
and Schramm \cite{BS2} is that $\lim_{n\to\infty} X_n$ exists and
is a
point $X_\infty\in\partial U$ almost surely, and furthermore its
distribution is nonatomic. This immediately implies that any bounded
measurable function $g: \partial U \to\R$ can be extended to a bounded
harmonic function $h: V \to\R$ by setting $h(v) = \E_v [
g(X_\infty)
 ]$, where for a vertex $x$ we write $\E_x$ for the expectation
w.r.t.
the random walk started at $x$. Since the distribution of $X_\infty$ is
nonatomic they deduce that a nonconstant bounded harmonic function
exists.
Note that since each vertex $v$ is in $\carr(P)$, we cannot have a vertex
on the boundary of $U$.

The main result of this paper is that there are no other bounded harmonic
functions, that is, any bounded harmonic function can be represented
this way.
Recall that a graph is one-ended if removal of any finite set of vertices
leaves only one infinite connected component.

\begin{theorem}
\label{T:mainmainthm}
Let $G=(V,E)$ be a transient bounded degree, one-ended triangulation and
let $P$ be a circle packing of $G$ with $\carr(P)=U$. Then for any
bounded harmonic function $h:V \to\R$ there exists a bounded measurable
function $g: \partial U \to\R$ such that $h(v)= \E_v  [
g(X_\infty)
 ]$.
\end{theorem}

For a vertex $x$, we write $\prob_x$ for the probability measure on
$G^{\mathbb{N}}$ of the Markov chain started at $x$. The measure space
$(G^{\mathbb{N}}, \I, \prob_x)$ is often called the \emph{Poisson boundary}
of the chain. The choice of $x$ does not matter much because the measures
$\P_x$ are absolutely continuous with respect to each other. As mentioned
before, there is a correspondence between bounded harmonic functions and
$L^\infty(G^{\mathbb{N}}, \I, \prob_x)$ and for that reason the
space of
bounded harmonic functions is sometimes also referred to as the Poisson
boundary. Theorem~\ref{T:mainmainthm} shows that if we circle pack $G$ in
$U$, then $\partial U$ is a representation of the Poisson boundary. More
precisely, let $f:G^{\mathbb{N}} \to\partial U$ be the measurable function
(defined $\P_x$-almost everywhere) $f(\{x_n\})= \lim x_n$ and let
$\mathcal{B} \subset\I$ be the pull back $\sigma$-algebra on
$G^{\mathbb{N}}$. Then $\mathcal{B}$ and $\I$ are in fact
equivalent, that is,
for any $A \in\I$ there exists $B \in\mathcal{B}$ such that the measure
of $A \bigtriangleup B$ is zero.

The Martin boundary \cite{Dynkin,Martin,woess} is another concept of a
boundary of a Markov chain, associated with the space of positive harmonic
functions. While the Poisson boundary is naturally defined as a measure
space, the Martin boundary is a topological space. It is well known (see
Chapter~24 of \cite{woess}) that the Poisson boundary may be obtained by
endowing the Martin boundary with a suitable measure. Hence, in
addition to
its intrinsic interest, the Martin boundary studied here will provide more
information and will yield Theorem~\ref{T:mainmainthm} rather abstractly.

An illustrative example of the difference between the boundaries is the
following. Let $G$ be the graph obtained from $\Z^3$ by connecting its root
to a disjoint one-sided infinite path. It is possible for a positive
harmonic function to diverge only along the path. Thus, the Martin boundary
will consist of two points (corresponding to the two ``infinities'' of
$G$), however, since the simple random walk has probability $0$ of staying
in the infinite path forever, and $\Z^3$ has no nonconstant bounded
harmonic functions, the Poisson boundary will have all its mass on one of
the points.

Let us formally define the Martin boundary. Let $x_0$ be an arbitrary fixed
root of $G$ and $M(x,y)$ be the \emph{Martin kernel}
\[
M(x,y) = \frac{ G(x,y)}{ G(x_0, y) },
\]
where $G(x,y) = \E_x[\# \mbox{ visits to $y$}]$ is the Green
function. For
any fixed $y$, the function $M(\cdot,y)$ is a positive function that is
harmonic everywhere except at $y$. Hence, if for some sequence $y_n$, such
that the graph distance between $x_0$ and $y_n$ tends to infinity, the
functions $M(\cdot, y_n)$ converges pointwise, then the limit is a positive
harmonic function on $G$. The \emph{Martin boundary} is defined to be
the set
$\cM$ of all such limit points, endowed with the pointwise convergence
topology.

A positive harmonic function $h:V \to\R$ such that $h(x_0)=1$ is called
\emph{minimal} if for any positive harmonic function $g$ such that $g(x)
\leq h(x)$ for all $x$, then $g=ch$ for some constant $c>0$. The minimal
functions are the extremal points of the convex set of positive harmonic
functions, normalized to have $h(x_0)=1$. By Choquet's theorem and
\cite{woess}, Theorem~24.8, it follows that any positive harmonic function
$h$ can be written as $h=\int g \,d\mu(g)$ for some measure $\mu$ depending
on $h$, and supported on the set of minimal harmonic functions. If we
normalize so that $h(x_0)=1$, then $\mu$ is a probability measure.

\begin{theorem}
\label{T:martinbdry}
Let $G=(V,E)$ be a transient bounded degree, one-ended triangulation and
let $P$ be a circle packing of $G$ with $\carr(P)=U$. Then:
\begin{longlist}[(1)]
\item[(1)] For a sequence $y_n\in V$ we have that $M(\cdot,y_n)$ converges
pointwise if and only if $y_n$ converges in $\R^2$ (in particular, the
limit only depends on $\lim y_n$).
\item[(2)] If $y_n \to\xi\in\partial U$, then $\lim M(\cdot, y_n)$ is a
minimal harmonic function.
\item[(3)] The map $\xi\mapsto\lim M(\cdot, y_n)$, where $y_n \to\xi$, is
a homeomorphism.
\end{longlist}
In particular, the Martin boundary is homeomorphic to $\partial U$.
\end{theorem}

The limit $\lim_{y\to\xi} M(\cdot, y)$ is denoted by $M_\xi$.
Thus, for any
positive harmonic function $h$ there is some measure $\mu$ on
$\partial U$,
so that $h = \int_{\partial U} M_\xi \,d\mu(\xi)$.

A similar characterization of the Poisson boundary of planar graphs in
terms of their square tiling was recently obtained by Georgakopoulos
\cite{Agelos}. His results allow him to characterize the Poisson boundary
for a somewhat more general set of graphs, namely, of bounded degree
\emph{uniquely absorbing} planar graphs. The analysis in this paper of random
walk via circle packings and other embeddings requires a completely
different set of tools and in return allows us to characterize the Martin
boundary with no additional cost.

\subsection{Good embeddings of planar graphs}

Recall that a \emph{proper embedding} of a planar graph is a map sending
the vertices to points in the plane and edges to continuous curves
connecting the corresponding vertices such that no two edges cross. If
each edge is mapped to a straight line we call it an \emph{embedding
with straight lines}. Given a circle packing of a graph $G$, we may
obtain such an embedding by mapping vertices to the corresponding
circle's center and edges to straight lines between the corresponding vertices.

We will prove our results for more general embeddings than the one
obtained from circle packing. The setting below has risen in the study
of critical 2D lattice
models and was formalized by Chelkak \cite{Chelkak}. Let $G=(V,E)$ be an
infinite, connected, simple planar graph together with an embedding with
straight lines. As before, we identify a vertex $v$ with its image in the
embedding. We write $|u-v|$ for the Euclidean distance between points in
the plane. For constants $D\in(1,\infty)$ and $\eta>0$, we say that the
embedding is \emph{$(D,\eta)$-good} if it satisfies:
\begin{longlist}[(a)]
\item[(a)]\emph{No flat angles.} For any face, all the inner angles are at most
$\pi- \eta$. In particular, all faces are convex, there is no outer face
and the number of edges in a face is at most $2\pi/\eta$.
\item[(b)]\emph{Adjacent edges have comparable lengths.} For any two adjacent
edges $e_1=(u,v)$ and $e_2=(u,w)$, we have that $|u-w|/|u-v| \in[D^{-1},
D]$.
\end{longlist}

We say that an embedding is \emph{good} if it has straight lines and it is
$(D,\eta)$-good for some $D,\eta$. A classical lemma of Rodin and Sullivan
\cite{RS} (known as the \emph{Ring lemma}) asserts that the ratio between
radii of tangent circles in a circle packing of a bounded degree
triangulation is bounded above and away from $0$. We immediately get the
following.

\begin{proposition}\label{P:CPisgood}
Any circle packing of a bounded degree triangulation is $(D,\eta)$-good
for some $D$ and $\eta$ that only depend on the maximum degree.
\end{proposition}

In a similar fashion to the circle packing setting, we define the carrier
of the embedding of $G$, denoted by $\carr(G)$, to be the union of all the
(closed)
faces of the embedding. Note that if $G$ is a one-ended triangulation,
then $\carr(G)$ is always an open simply connected set in the plane.
Lastly, suppose that the edges of the graph are equipped with positive
weights $\{w_e\}_{e\in E}$ and consider the weighted random walk $\{
X_n\}$
defined by $P(X_1 = u | X_0 = v) = w_{(v,u)}/w_v$ for any edge $(u,v)$,
where $w_v = \sum_{u: u \sim v} w_{(u,v)}$. A function $h:V\to\R$ is
harmonic with respect to the weighted graph when
%
\begin{equation}
\label{harmonicdef} h(v) = \sum_{u: u \sim v} \frac{w_{(u,v)}} {w_v}
h(u),
\end{equation}
or in other words, when $h(X_n)$ is a martingale. The general version of
Theorems~\ref{T:mainmainthm} and~\ref{T:martinbdry} is now stated in a straightforward
manner.

\begin{thms*}\label{thms1}
Let $G=(V,E)$ be
a bounded degree planar graph with a good embedding with straight lines
such that $\carr(G)=U$. Assume that $G$ is equipped with positive edge
weights bounded above and away from $0$. Then the conclusions of
Theorems \ref{T:mainmainthm} and \ref{T:martinbdry} hold verbatim.
\end{thms*}

Theorems \ref{T:mainmainthm} and \ref{T:martinbdry} are immediate corollaries of this
statement together with Proposition~\ref{P:CPisgood}.

\subsection{Harmonic measure and exit time of discrete balls}

In the following, let $G=(V,E)$ be a planar graph with a good
embedding. A
discrete domain $S$ is a subset of $V$ along with the induced edges
$E(S) =
(S \times S) \cap E$. The boundary of $S$, denoted $\partial S$ is the
external vertex boundary, that is, all vertices not in $S$ with a
neighbor in
$S$. For $u \in\R^2$ we denote by $\Beuc(u,r)$ the Euclidean ball
$\{ y
\in\R^2 : |u-y| \leq r\}$ of radius $r$ centred at $u$, and the discrete
Euclidean ball $\Veuc(u,r)$ is the vertex set
\[
\Veuc(u,r) = V \cap\Beuc(u,r).
\]

As before, assume that the edges are equipped with positive weights and
consider the weighted random walk $\{X_n\}$. For $A \subset V$ let
$\tau_A$
be the first hitting time of $A$, that is, $\tau_A = \min\{ n : X_n
\in A
\}$ or $\infty$ if $A$ is never hit. The following two theorems answer a
question recently posed by Chelkak (\cite{Chelkak}, page 9).

\begin{theorem}
\label{T:exit_arg}
For any positive constants $D,\eta$ there exists $c=c(D,\eta)>0$ with the
following. Assume that $G$ is a graph with a $(D,\eta)$-good embedding,
and all edges weights in $[D^{-1}, D]$. Then for any vertex $u$, any $r
\geq0$ such that $\Beuc(u,r) \subset\carr(G)$ and any closed interval
$I\subset\R/(2\pi\Z)$ of length $\pi-\eta$ we have
\[
\prob_u \bigl( \arg( X_{T_r} - u ) \in I \bigr) \geq c,
\]
where $T_r = \tau_{\partial\Veuc(u,r)}$ is the first exit time from
$\Veuc(u,r)$.
\end{theorem}

Note that for smaller intervals of arguments the statement above may be
false; for example, the left-hand side is $0$ if $\partial\Veuc(u,r)$
contains no vertex in these directions.

For a vertex $u\in V$, we denote its \emph{radius of isolation} by $r_u =
\min_{V\setminus\{u\}} \{|u-v|\}$. We use $f\asymp g$ when there is some
$C=C(D,\eta)$ so that $C^{-1} g \leq f \leq C g$.

\begin{theorem}
\label{T:exit_time}
For any positive constants $D,\eta$ there exists $C=C(D,\eta)\geq1$ with
the following. Assume that $G$ is a graph with a $(D,\eta)$-good
embedding and all edges weights are in $[D^{-1}, D]$. Then for any
vertex $u$ and any $r \geq r_u$ with $\Beuc(u, C r)\subset\carr(G)$
we have
\[
\E_u \sum_{t=0}^{T_r}
r^2_{X_t} \asymp r^2,
\]
where $T_r = \tau_{\partial\Veuc(u,r)}$ is the first exit time from
$\Veuc(u,r)$.
\end{theorem}

The reader may wonder why we require $\Beuc(x_0,Cr) \subset\carr(G)$,
while the theorem only talks about the time to exit the smaller
$\Beuc(x_0,r)$. This is an artifact of our proof, and the stronger
requirement can indeed be removed. This requires showing that it is
possible to ``extend'' the embedding to a good embedding of a larger graph
with carrier $\R^2$. This is indeed possible, and we plan to address this
in a future paper.

\subsection{About the proofs and the organization of the paper}

We would like to compare the random walk on a well-embedded graph to
Brownian motion, and certainly our results above justify such a
comparison. However, the simple random walk on a good embedding can behave
rather irregularly. For example, its Euclidean trajectory is not a
martingale and can have a local drift. The random walk is also much slower
when traversing areas of short edges compared to areas of longer edges. To
fix the second problem, we could study the variable speed random walk which
waits at each vertex an amount of time comparable to $r_x^2$, or to the
area of one of the faces containing the vertex (a good embedding guarantees
that all faces sharing a vertex have comparable area). Instead, we use the
\emph{cable process} on the graph, which can be thought of as Brownian
motion on the embedding (see Section~\ref{sec-cableprocess}). The vertex
trajectory of this process has the same distribution as the simple random
walk, so the harmonic measures do not change.

A central step in this work is showing that well-embedded graphs satisfy
volume doubling and a Poincar\'e inequality with respect to the Euclidean
metric (rather than the graph metric). This is done in
Section~\ref{sec-cableprocess}. The work of Sturm \cite{Sturm} (which
applies in
the very general setting of local Dirichlet spaces) then enables us to
obtain various corollaries: an elliptic Harnack inequality
(Theorem~\ref{cableehi}) and heat kernel estimates (see Theorem~\ref
{T:heat}). These
already give us enough control to prove Theorems \ref
{T:exit_arg} and \ref{T:exit_time} in
Section~\ref{sec-harmonic}.

To prove Theorems \ref{T:mainmainthm} and \ref{T:martinbdry}, we require a
boundary Harnack
inequality (see Theorem~\ref{bhp}). Roughly speaking, this states that two
positive harmonic functions that vanish on most of the boundary of the
domain do so in a uniform way. In our setting, the boundary Harnack
inequality is a consequence of the volume doubling and Poincar\'e
inequality, as shown in \cite{LS}, following an argument of Aikawa
\cite{A}
that originates in the work of Bass and Burdzy \cite{BB}. Given the
boundary Harnack inequality, it is possible to prove Theorem~\ref
{T:mainmainthm}
by constructing an explicit coupling between two random walks starting at
two different points conditioned to converge to some $\xi\in\partial U$
(by conditioning that the random walk is swallowed in a small neighborhood
and taking a weak limit) so that with probability $1$ their traces
coincides after a finite number of steps. This coupling is constructed by
showing that for any annulus of constant aspect ratio around $\xi$ the
conditioned random walks have a positive chance to meet.

We do not use this proof approach and instead use the more succinct
approach of Aikawa \cite{A}. His argument (following Jerison and Kenig
\cite{JK}) shows how the characterization of the Martin boundary of
Brownian motion on a uniform domain follows from the boundary Harnack
principle. Our argument in Section~\ref{sec-martin} is very similar to
\cite{A}
except for the complication that our process is not a martingale. Thus, a
separate argument is necessary to show the convergence of the random walk
to the boundary and that the distribution of the limit is nonatomic.

\section{Preliminaries}\label{S:prelim}

We begin with some geometric consequences of having a good embedding. In
this section, we assume that we are given a $(D,\eta)$-good embedding
of a
graph $G$.

\begin{figure}[b]

\includegraphics{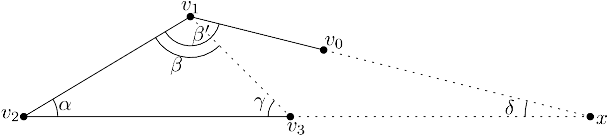}

\caption{No thin acute angles: In a good embedding, the angle $\alpha
$ must be
at least $D^{-1} \sin(\eta/2)$.}
\label{fig.noacute}
\end{figure}

\begin{lemma}[(No thin acute angles)]
\label{L:noacute}
The angle between any two adjacent edges is at least $D^{-1}
\sin(\eta/2)$.
\end{lemma}

\begin{pf}
Let $\alpha$ be the angle between three consecutive vertices on a face
$v_1, v_2,v_3$ such that the edge $[v_1,v_2]$ is not longer than the edge
$[v_2,v_3]$. By convexity, the triangle $v_1,v_2,v_3$ is contained in the
face. Let $\beta, \gamma$ be the angles $\angle v_2 v_1 v_3$ and
$\angle
v_2 v_3 v_1$, respectively. By our assumption, we learn that $\beta
\geq
\gamma$, hence $\gamma\leq\pi/2$. See Figure~\ref{fig.noacute}.

\begin{figure}[b]

\includegraphics{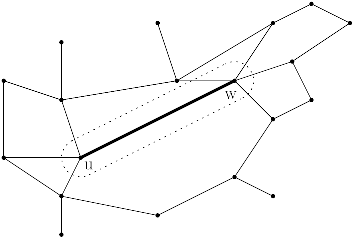}

\caption{The sausage lemma: no edge can intersect the marked ``sausage''
with width $c|e|$.}
\label{fig.sausage}
\end{figure}

If $\alpha\geq\eta/2$, then we are done since $\eta/2 \geq D^{-1}
\sin(\eta/2)$. Otherwise, let $v_0$ be the vertex before $v_1$ on the
face (if the face is a triangle, then $v_0=v_3$). Let $\beta'$ be~the
angle $\angle v_0 v_1 v_2$ so that $\beta'\geq\beta$, and by (b) we have
$\beta' \leq\pi- \eta$. Let $x\in\R^2$ be the~meeting point of
the ray
emanating from $v_1$ toward $v_0$ and the ray emanating from $v_2$
toward $v_3$ (since $\alpha+ \beta' < \pi$ these rays must intersect
and the intersection point $x$ must be on the same side of the infinite
line through $v_1, v_2$ as $v_0$ and $v_3$). Let $\delta$ be the angle
$\angle v_1 x v_2$. We have that $\delta=\pi- \beta' - \alpha$ hence
$\eta/2 \leq\delta\leq\gamma\leq\pi/2$. Hence, by the law of sines
\[
\alpha\geq\sin(\alpha) = { |v_1 x| \sin(\delta) \over|v_1 v_2| } \geq{ |v_1 v_0| \sin(\delta) \over|v_1 v_2|}
\geq D^{-1} \sin(\eta/2),
\]
where $|v_1 x| \geq|v_1 v_0|$ since by convexity $v_0$ and $v_1$ are on
the same side of the infinite line passing through $v_2$ and $v_3$.
\end{pf}

\begin{lemma}[(Sausage lemma)]
\label{L:sausage}
There exists $c=c(D,\eta)>0$ such that if $e,f$ are nonadjacent edges
then $d(e,f) \geq c|e|$, where $|e|$ and $d(\cdot,\cdot)$ are Euclidean
length and distance. In particular, any vertex $u\in V \setminus e$ is
of distance at least $c|e|$ from $e$.
\end{lemma}

\begin{pf} Write $e=\{v,w\}$ and see Figure~\ref{fig.sausage}. Let
$v_1,v_2$ be two consecutive neighbors
of $v$. Because the angle $\angle v_1 v v_2$ is at most $\pi-\eta$, one
of the angles $\angle v_1 v_2 v$ or $\angle v_2 v_1 v$ is at least
$\eta/2$ and by (c) both $|v v_1|$ and $|v v_2|$ are at least $D^{-1}
|e|$. Hence, the distance between $v$ and the line through $v_1$ and
$v_2$ is at least $D^{-1} \sin(\eta/2)|e|$. We conclude that there are
no points of $X$ inside a ball around $v$ of radius $D^{-1}
\sin(\eta/2)|e|$ except for the edges emanating from $v$ and the same
holds for $w$.

\begin{corollary}
\label{L:chorizo}
There exists $c=c(D,\eta)>0$ such that for any edge $e$ and vertex $u
\notin e$ and $r>0$ we have that if $e$ intersects $\Beuc(u,c r)$,
then $e$
is contained in $\Beuc(u,r)$.
\end{corollary}

Next, consider one of the two faces containing $e$ and let $v_1$ be the
neighbor of $v$ in the face that is not $w$ and similarly $w_1$ be the
neighbor of $w$ in the face that is not $v$ (if the face is a triangle,
then $w_1=v_1$). By Lemma~\ref{L:noacute} the angles $\angle v_1 v w$ and
$\angle v w w_1$ are at least $D^{-1} \sin(\eta/2)$ and by condition (b)
these angles are at most $\pi-\eta$. Hence, by condition (c), the face
contains a trapezoid in which $e$ is a base and the two sides are
sub-intervals containing $v$ and $w$ of the edges $(v,v_1)$ and $(w,w_1)$,
respectively, and of height at least $D^{-1} |e| \sin(D^{-1}
\sin(\eta/2))$.
\end{pf}

\begin{lemma}\label{L:face_diam}
There exists a constant $C=C(D,\eta)<\infty$ such that if an edge $e$ is
contained in a face $f$, then $\diam(f) \leq C|e|$, and
$ C^{-1} \leb(f) \le|e|^2 \leq C
\leb(f)$, where $\leb(f)$ is the usual Lebesgue area measure in $\R^2$.
\end{lemma}

\begin{pf}
Since external angles in a polygon add up to $2\pi$, the number of sides
of a face is at most $2\pi/\eta$, and since consecutive sides have length
ratio at most $D$, any two sides of a face have ratio at most $C$, and
the diameter of the face is at most some constant times the shortest
edge.

The relation to the area of $f$ follows from Lemma~\ref{L:noacute} and that
edges adjacent to $e$ have comparable lengths.
\end{pf}

Most of our arguments will take place in the metric space $(X,d_0)$ defined
as follows. For an edge $(u,v)\in E$, write $[u,v]$ for the closed line
segment in the plane from $u$ to $v$. We put
\[
X = \bigcup_{(u,v)\in E} [u,v],
\]
and let $d_0$ be the shortest path distance in $X$. For $x\in X$ and $r>0$
we write $B_{d_0}(x,r)$ for the ball $\{y \in X: d_0(x,y) \leq r\}$. An
idea that we will use frequently is to take a curve in $\R^2$ with
some useful properties and modify it slightly to get a curve in $X$ with
similar properties.

\begin{proposition}\label{P:bilipschitz}
Assume that $\carr(G)=U$. There exists a constant $C_1=C_1(D,\eta)$ such
that for any $x,y\in X$ we have
%
\begin{equation}
\label{e:d0equiv} |x-y| \leq d_0(x,y) \leq C_1 |x-y|.
\end{equation}
\end{proposition}

\begin{pf}
Since $d_0(x,y)$ is the Euclidean length of the shortest path in $X$
between $x$ and $y$ the inequality $|x-y| \leq d_0(x,y)$ is obvious.

To prove the other inequality, we first prove the assertion for $x$ and
$y$ that are on the same face. If $x$ and $y$ are on the same edge then
$d_0(x,y) = |x-y|$. If $x$ and $y$ are on two different edges that
share a vertex $v$, then since the angle at $v$ is bounded away from $0$
(by Lemma~\ref{L:noacute}) we deduce by the law of sines on the triangle
$x,v,y$ that $d_0(x,y) \leq|x-v| + |y-v| \leq C |x-y|$.
Lastly, when $x$ and $y$ are on two edges of the same face not sharing a
vertex, Lemma~\ref{L:sausage} immediately gives that $|x-y|$ is at
least $c|e|$, where $e$ is some edge that contains
$x$. Lemma~\ref{L:face_diam} gives that $|e|$ is at least a constant multiple
times the diameter of the face and the assertion follows.

Finally, when $x$ and $y$ are not on the same face let $[x,y]$ be the
straight segment connecting $x$ and $y$ and let $x=x_0,x_1,\ldots,x_k=y$
be the points on $[x,y]$ where the segment intersects $X$, so that $x_i$
and $x_{i+1}$ are on the boundary of some face for all $i=0,\ldots, k-1$.
Then $d_0(x_i,x_i+1) \leq C |x_i-x_{i+1}|$ and summing over $i$ finishes
the proof of the lemma.
\end{pf}

Consequently,
the completion of $(X,d_0)$ is $\overline X = X\cup\partial U$ with
the topology
induced from $\R^2$, and \eqref{e:d0equiv} extends to the space
$\overline X$.

\begin{lemma}[($X$ is inner uniform)]
\label{L:inneruniform}
Assume that $G$ has a $(D,\eta)$-good embedding and that
$\carr(G)=U$. There exist constants $C=C(D,\eta)<\infty$ and
$c=c(D,\eta)>0$ such that for any $\xi_1, \xi_2
\in\partial U$ with $\xi_1 \neq\xi_2$ there exists a continuous curve
$\Gamma:[0,L] \to\overline{X}$ such that the following holds:
\begin{longlist}[(1)]
\item[(1)]$\Gamma$ is parametrized by length, that is,
$\operatorname{length}(\Gamma[0,t]) = t$ for all $t\in[0,L]$.
\item[(2)]$\Gamma(0) = \xi_1$ and $\Gamma(L) = \xi_2$.
\item[(3)]$L \leq C d_0(\xi_1,\xi_2)$.
\item[(4)] For any $t \in(0,L)$ we have
\[
d_0\bigl(\Gamma(t), \partial U\bigr) \geq c \min(t, L-t).
\]
\end{longlist}
\end{lemma}

\begin{pf}
Consider a circle orthogonal to $U$ through $\xi_1,\xi_2$, and the
continuous curve $\gamma$ which is the arc from $\xi_1$ to $\xi_2$
in that
circle. Let $(\ldots, x_{-2},\break x_{-1},x_0,x_1, x_2,\ldots)$ be the set
$\gamma\cap X$ with the order induced by $\gamma$. Two consecutive
points $x_i$ and $x_{i+1}$ are on the same face, so write $\Gamma_i$ for
a piecewise linear curve on the boundary of that face connecting $x_i$ to
$x_{i+1}$ in the shorter way according to $d_0$. Let $\Gamma$ be the
concatenation of $(\Gamma_i)_{\{-\infty< i < \infty\}}$,
parametrized by
arc length (which we shall see below is finite). Note that $\Gamma$ might
not be a simple curve, which does not cause any difficulty. Thus, (1)
holds. Let us show that $\Gamma$ satisfies requirements (2)--(4).

We first note that by Lemma~\ref{L:sausage} we have that there exists $C>0$
such that for any face $f$
%
\begin{equation}
\label{facebdry} \max_{z \in\partial f} d_0(z, \partial U) \leq
C \min_{z \in\partial f} d_0(z, \partial U),
\end{equation}
where by $z \in\partial f$ we mean that $z\in X$ is on one of the edges
encompassing $f$. Now, it is clear that $x_k \to\xi_1$ when $k\to
-\infty$ and $x_k \to\xi_2$ when $k \to\infty$, so by \eqref{facebdry}
we get that $\Gamma$ satisfies requirement (2).

Next, we have that $\operatorname{length}(\Gamma_i) = d_0(x_i,x_{i+1})$ since $x_i$ and
$x_{i+1}$ are on the same face $f$, and since the shortest curve between
two points on the boundary of a convex face $f$ that does not enter the
face is along its boundary. Thus, we have $\operatorname{length}(\Gamma) = \sum_i
d_0(x_i,x_{i+1}) \leq C |\xi_1-\xi_2|$ by Proposition~\ref
{P:bilipschitz}, so
requirement (3) holds. Lastly, (4) holds immediately for the points
$x_k$, and by \eqref{facebdry} we obtain this for any point on $\Gamma$.
\end{pf}

\section{The cable process}
\label{sec-cableprocess}

It will be convenient for us to obtain useful
estimates using results of Sturm \cite{Sturm}. For that, we need to
introduce the \emph{cable process} which can be thought of as Brownian
motion on the embedding of $G$. (See, e.g., \cite{Var}.)
Recall that $(w_e)$ are edge weights on $G$, which are bounded above
and away from~0.
An intuitive description of the process is
as follows: Let $x$ be a vertex and $e_1, \ldots, e_k$ the edges emanating
from it, and let $\{W_t\}_{t \geq0}$ be standard Brownian motion. It is
well known that $W_t$ can be decomposed into countably many excursions in
which $W_t \neq0$. For each such excursion, we choose the edge $e_i$ with
probability proportional to $w_{e_i} |e_i|$ for $i=1,\ldots,k$ and embed
the excursion on the edge $e_i$. We stop when we hit one of the neighbors
$x_1, \ldots, x_k$ of $x$, and continue from this neighbor using the
strong Markov property.
Thus, this process is a standard Brownian motion on the edges (or ``cables''),
and behaves like a Walsh Brownian motion (see \cite{Wal}) at the vertices.

Note that it is possible for this process to
``explode'', or visit infinitely many vertices in finite time, and indeed
this does happen almost surely in the transient setting.

Before defining the process formally, let us state two useful properties
that will make the connection to the discrete time weighted random walk
evident. Denote by $Z_t$ the process and let $T$ be the hitting time of
$\{x_1,\ldots, x_k\}$. Then for $1 \leq i \leq k$, we have (see
\cite{Folz}, Theorem~2.1)
%
\begin{equation}
\label{rwtrace} \prob_x ( Z_T = x_i ) =
\frac{w_{e_i} }{ \sum_{i=1}^k w_{e_i} },
\end{equation}
that is, the process $Z_t$ observed on
vertices has the same trace as the simple random walk (we ignore the
uncountably many times it visits each $x$ before proceeding to one of its
neighbors). Also, in our setting, there exists a constant $C=C(D,\eta)>0$
such that (see \cite{Folz}, Theorem~2.2)
%
\begin{equation}
\label{timetomove} C^{-1} \leq\frac{\E_x T}{r_x^2} \leq C,
\end{equation}
where $r_x$ is the length of the shortest edge touching $x$ (and so is
comparable to the length of any edge touching $x$). Intuitively, the
process $Z_t$ behaves like the variable speed random walk that waits
roughly $r_x^2$ time at vertex $x$ before proceeding.

The construction based on excursions can be made precise. However,
it is easier to define the cable process via the methods of
Dirichlet forms (see \cite{fot}).
Let $(\overline X,d_0)$ be the compact metric space defined in
Section~\ref{S:prelim}.
For an edge $(u,v)$, write $dx$ for Lebesgue measure on $[u,v]$, and
define a measure $m$ on $\overline X$ by taking
\[
m(dx) = \sum_{(u,v)\in E} |u-v| w_{uv} \,dx.
\]

\begin{lemma}\label{mXfinite}
We have $m(\overline X)<\infty$.
\end{lemma}

\begin{pf}
By Lemma~\ref{L:face_diam}, the measure of any edge $e$ is less than
$C\leb( f_1 \cup f_2)$,
where $f_1, f_2$ are the two faces containing $e$. Since $\carr(G)
=U$, it follows that
$m(\overline X) \le2C \leb(U)$.
\end{pf}

We say that a function $f$ on $\oX$
is piecewise differentiable if it is continuous at each vertex, and is
differentiable w.r.t. the length measure on every edge. (We require
that the one sided derivatives exist at the end of each edge.)
The derivative $f'$ depends on the direction, and only makes
sense if we fix a direction for every edge. However, $f' g'$ is
well defined for differentiable $f,g$ and does not depend on choosing a
direction on the edges. Let $\sD_0$ be the space of differentiable
functions $f$ on $\oX$
such that $f'$ is continuous on each edge, and $|f'|$ is bounded.
For $f,g \in\sD_0$, let
\begin{eqnarray*}
d\Gamma(f,g) (dx) &=& \sum_{(u,v)\in E}
1_{[u,v]}(x) f'(x) g'(x) |u-v|
w_{uv} \,dx,
\\
\sE(f,g) &=& \sum_{(u,v)} \int_{[u,v]}
f'(x) g'(x) |u-v| w_{uv} \,dx = \int
_{\oX} \,d\Gamma(f,g).
\end{eqnarray*}
Let $\sD$ be the completion of $\sD_1$ with respect to the norm
\[
\Vert f\Vert _{\sE_1} = \biggl( \sE(f,f) + \int_{\oX}
f^2 \,dm \biggr)^{1/2}.
\]
It is straightforward to verify that the bilinear form $(\sE, \sD_0)$
is closed and Markov
(see \cite{fot}, page 4 and \cite{CF}, Section~2.2), so that $(\sE,
\sD)$ is a Dirichlet form.
Since $\sD_0$ is dense in $C(\oX)$, the continuous functions on $\oX$,
$(\sE, \sD)$ has a core and is thus a regular Dirichlet form. The
associated strong Markov process $Z$ is the cable process on $\oX$. We remark
that with this construction the functions $f$ in the domain $\sD$ do
not vanish
on the boundary $\partial U$, so that the process $Z$ is conservative
and will
reflect from the boundary after its first hit.

The space $(\oX,m,\sE)$ has an intrinsic metric associated with it (see
\cite{Sturm}), where the distance between $x,y$ is given by
\[
\sup \bigl\{ f(x)-f(y) : f \in\sD \mbox{ and } \bigl|f'\bigr| \leq1 \bigr\}.
\]
In our case, it is clear that this metric coincides with $d_0$ defined
above. We will show that this space is doubling and has a weak Poincar\'e
inequality.

\subsection{Doubling}

\begin{lemma}\label{strongdoubling}
Let $X$ be a good embedding of some graph. There exists a integer
$M=M(D,\eta)>0$ such that for any $x\in X$ and any $r>0$ there exists\vadjust{\goodbreak}
$x_1, \ldots, x_M \in X$ such that
\[
B_{d_0}(x,2r) \subset\bigcup_{i=1}^M
B_{d_0}(x_i,r).
\]
\end{lemma}

\begin{pf}
This is an easy consequence of the equivalence between $d_0$ and the
Euclidean metric. We have that $B_{d_0}(x,2r) \subset\Beuc(x,2r)$. Let
$C_1$ be the constant from Proposition~\ref{P:bilipschitz}. The
Euclidean ball can
be covered by $M$ Euclidean balls $\Beuc(y_i,r/2C_1)$ of radius $r/2C_1$
for some $M\asymp C_1^2$. For each $i$, if $\Beuc(y_i,r/2C_1)$
intersects $X$ and $x_i$ is an arbitrary point in the intersection then
$\Beuc(y_i,r/2C_1) \cap X \subset B_{d_0}(x_i,r)$. Otherwise, we ignore
this ball. So the collection of balls $B_{d_0}(x_i, r)$ covers
$B_{d_0}(x,2r)$.
\end{pf}

Recall the radius of isolation $r_u$ defined for $u\in V$ as the distance
to the nearest vertex $v\neq u$. We extend this to $x\in X$ by letting
$r_x$ be the length of the edge containing $x$ when $x\in X\setminus
V$, and
setting $r_x=0$ if $x \in\partial U$.
Note that if $r \in(0,1)$ and $x \in\oX$ then we have $\leb( \Beuc
(x,r)) \ge c r^2$.

\begin{lemma}\label{doubling}
For any $x \in\oX$ and $r\in(0,1)$,
we have
%
\begin{equation}
\label{e:mBbound} m\bigl(B_{d_0}(x, r)\bigr) \asymp r\cdot(r\vee
r_x).
\end{equation}
In particular, there exists a constant $C>0$ such that for $r>0$
\[
m\bigl(B_{d_0}(x,2r)\bigr) \leq C m\bigl(B_{d_0}(x,r)\bigr).
\]
\end{lemma}

\begin{pf}
First, we assume that $x$ is a vertex.
When $r \leq r_x$, we have that $m(B_{d_0}(x,r)) \asymp
r_x r$, because the degrees of $G$ are bounded and adjacent edges have
comparable length. So it suffices to prove that $m(B_{d_0}(x,r)) \asymp
r^2$ when $r \geq r_x$.

By Lemma~\ref{L:face_diam}, for an edge $e$ we have $m(e) \asymp|e|^2
\leq
C\leb(f)$ where $f$ is a face containing $e$. Since each face has a
bounded degree, and since faces intersecting $B_{d_0}(x,r)$ are fully
contained in $\Beuc(x,Cr)$ we find $m(B_{d_0}(x,r)) \leq C r^2$,
where $C=C(D,\eta)<\infty$.
For the lower bound, if $r_x \le r \le Cr_x$ then
$m(B_{d_0}(x,r)) \ge r_x^2 \ge C^{-2} r^2$. If $r>Cr_x$
by Corollary~\ref{L:chorizo} and Proposition~\ref{P:bilipschitz}, the union
of the set of faces $f$
adjacent to edges $e$ contained in $B_{d_0}(x,cr)$ contains
$U \cap\Beuc(x,c^2 r)$, giving the required lower bound.

If $x \in X$ is not a vertex, then let $y$ be a closest vertex to $x$,
and let $s=d_0(x,y)$.
If $r\le s$ then $m(B_{d_0}(x,r))=2 r r_x$, which equals $2 r( r \vee
r_x)$ since
$r_x \ge s$.
Now suppose that $r>s$. Then $B_{d_0}(x,r) \subset B_{d_0}(y,2r)$,
and since $r_y \asymp r_x$, this gives the upper bound in \eqref{e:mBbound}.
If $r \in[s,2s]$, then $m(B_{d_0}(x,r)) \ge\frac{1}2 r r_x$, while if
$r>2s$ then $B_{d_0}(y,r/2) \subset B_{d_0}(x,r)$, so in either case we
have the lower bound.


Finally, if $x \in\partial U$ and $r>0$ then we can find a vertex $y$
such that
$d_0(x,y) < r/2$ and $r_y < r/2$, and use the bounds for
$m(B_{d_0}(y,\cdot))$.
\end{pf}

\subsection{Poincar\'e inequality}

The strong Poincar\'e inequality states that for any
differentiable $f$ on $B=B_{d_0}(x,r)$ we have
%
\begin{equation}
\label{strongpi} \int_B \bigl\llvert f(x) - \bar{f}\bigr
\rrvert ^2 m(dx) \leq C r^2 \int_B
\bigl\llvert f'(x)\bigr\rrvert ^2 m(dx),
\end{equation}
where $\bar{f} = \frac{1}{m(B)} \int_B f m(dx)$ is the mean of $f$. A
well-known technique due to Jerison (see \cite{Jerison}, Section~5 and also
\cite{SC}, Section~5.3, for a simpler proof) shows that for spaces
satisfying the doubling property, this follows from the weak Poincar\'e
inequality which we now prove.

\begin{theorem}[(Weak Poincar\'e inequality)]
\label{T:weakpi}
There exist positive constants $C=C(D,\eta)$ and $C'=C'(D,\eta)$ such
that for any $x_0 \in X$ and $r \in(0,1)$
with $\Beuc(x_0,Cr) \subset\carr(G)$,
and all $f$ piecewise differentiable on $B_{d_0}(x_0,Cr)$ we have
\[
\int_{B_{d_0}(x_0,r)} \bigl\llvert f(x) - \bar{f}\bigr\rrvert
^2 m(dx) \leq C' r^2 \int
_{B_{d_0}(x_0,Cr)} \bigl\llvert f'(x)\bigr\rrvert
^2 m(dx).
\]
\end{theorem}

We begin with the following lemma.

\begin{lemma}\label{L:close_to_leb}
Let $x$ have law $\frac{m(dx)}{m(B)}$ on some set $B\subset X$, and
conditioned on $x$ let $\hat x$ be uniform in the union of the two faces
incident to the edge containing $x$. Then there is some constant
$C=C(D,\eta)$ so that the law of $\hat x$ is bounded by
$\frac{C\leb}{m(B)}$, where $\leb$ is the usual Lebesgue measure on
$\R^2$.
\end{lemma}

\begin{pf}
For an edge $e$ of $X$, we have that $\P(x\in e) = \frac{m(e\cap B)}{m(B)}
\leq\frac{|e|^2}{m(B)}$ (with equality holding when $e \subset B$). If
$e$ is incident to some face $f$ then the conditional contribution to the
density of $\hat x$ in $f$ is at most $1/\leb(f)$, and so the density on
a face $f$ surrounded by edges $e_1,\ldots,e_k$ is at most $\frac{1}{m(B)}
\sum\frac{|e_i|^2}{\leb(f)}$. The number of edges surrounding a face is
at most $2\pi/\eta$, and the square of each is comparable to the area of
$f$, giving the claim.
\end{pf}

\begin{pf*}{Proof of Theorem~\ref{T:weakpi}}
We let $B=B_{d_0}(x_0,r)$. Let $x,y$ be independent points chosen in $B$
with law $\frac{m(dx)}{m(B)}$. We start with the simple identity
\[
\int_B |f(x)-\bar{f}|^2 m(dx) =
\frac{m(B)}{2} \E\bigl|f(x)-f(y)\bigr|^2,
\]
that follows from expanding. Let $\gamma=\gamma_{xy}$ be some (possibly
random) path in $X$ between $x$ and $y$, then $f(y)-f(x) = \int_\gamma
f'(z) \,dz$, where $dz$ is the length element along $\gamma$. Applying
Cauchy--Schwarz gives
\[
\bigl|f(x)-f(y)\bigr|^2 \leq|\gamma| \int_\gamma\bigl|f'(z)\bigr|^2
\,dz,
\]
where $|\gamma|$ denotes the length of $\gamma$.

To construct the path, we need to consider two cases. If
$B_{d_0}(x_0,r)$ contains a single vertex of the graph, then the graph
in $B_{d_0}(x_0,r)$ is a star, and there is an obvious choice of path
$\gamma_{xy}$. If there are at least two vertices, we proceed as
follows.
Let $\hat x$ (resp., $\hat y$) be uniformly chosen in the union of the
two faces of $X$ incident to $x$ (resp., to $y$). The straight line
segment $\hat x \hat y$ begins at a face containing $x$, ends at a face
containing $y$, and possibly passes through some other faces in
between. As in the proof of Proposition~\ref{P:bilipschitz}, we can
approximate this
line segment by a path $\gamma_{xy}$ in $X$, which stays in the
boundaries of faces crossed by the line segment.

We first observe that due to Lemma~\ref{L:sausage} we have that $\hat
x,\hat y
\in B_0 := \{u \in\R^2: |u-x_0|<C_0 r\}$ for some $C_0=C_0(D,\eta
)\geq
1$, and $B_0 \subset\carr(G)$ by our assumptions. By increasing
$C_0$, we can guarantee that $\gamma_{xy}$ also does not leave $B_0$, and
as in Proposition~\ref{P:bilipschitz}, we have $|\gamma_{xy}|\leq
Cr$. We
shall see below that $\P(z\in\gamma_{xy}) \leq\frac{C \rho_z}{r}$ where
$\rho_z$ is the length of the edge containing $z$ (we neglect the measure
$0$ set of vertices). Given that we conclude the proof as follows:
\begin{eqnarray*}
\int_B \bigl|f(x)-\bar{f}\bigr|^2 m(dx) &\leq& C m(B) \E
\biggl[ r \int_\gamma\bigl|f'(z)\bigr|^2 \,dz
\biggr]
\\
&\leq& C r m(B) \int_{z\in B_0} \frac{C\rho_z}{r}
\bigl|f'(z)\bigr|^2 \,dz
\\
&\leq& Cr^2 \int_{B_0} \bigl|f'(z)\bigr|^2
m(dz),
\end{eqnarray*}
since $\rho_z\, dz = m(dz)$, and $m(B) \leq Cr^2$.

To bound the probability that $z\in\gamma_{xy}$, note that the faces
incident to $z$ are contained in $\{u \in\R^2: |u-z|\leq
C_1\rho_z\}$. Let $A$ be the event that the segment $\hat x\hat y$
intersects a face incident to $z$, and $A'$ the event that the segment
passes within distance $C_1\rho_z$ of $z$, so that $A\subset A'$. Let
$\hat m$ be the law of $\hat x$ and $\hat y$. By Lemma~\ref{L:close_to_leb},
we have that
\[
\hat m \leq\frac{C}{m(B)} \leb\leq\frac{C_2}{r^2} \leb,
\]
and $\supp\hat m \subset B_0 \subset\{|\hat x-z|\leq2C_0 r\}$. We have
now
\begin{eqnarray*}
\P(z\in\gamma_{xy}) &\leq&\E1_A \leq\E1_{A'} =
\iint1_{A'} \,d\hat m \times d\hat m
\\
&\leq&\int_{|\hat x-z|\leq2C_0 r} \int_{|\hat y-z| \leq2C_0 r}
1_{A'} \biggl(\frac{C_2}{r^2} \biggr)^2 \,d\hat x \,d\hat y.
\end{eqnarray*}
By scaling and translating this is $(C_2/2C_0)^2$ times the
probability that the segment between two uniform points $u,v\in U$ passes
within $C_1\rho_z/r$ of the origin. For such $u,v$, the distance between
the segment and the origin is a continuous random variable with finite
density, so the distance is at most $C_1\rho_z/r$ with probability
at most $C\rho_z/r$.
\end{pf*}

\subsection{Heat kernel estimates}

Finally, we are able to deduce estimates for the heat kernel of the cable
process on $X$. Let $q_t(x,y)$ denote the heat kernel for the Markov
process $\{Z_t\}_{t \geq0}$ associated with $(\oX,m,\sE)$, that is,
$q_t(x,\cdot)$ is the density (with respect to $m$) of $Z_t$
conditioned on
$Z_0=x$. For a set $A\subset\oX$, we let $q_t^A(x,y)$ denote the heat kernel
for the process killed when it exits $A$. (If $A= \oX$, then $q^A$ is
just the
unkilled heat kernel.)

\begin{theorem}\label{T:heat}
There exists constants $c,C$ depending only on $D,\eta$ such that for any
$x_0 \in X$ and $r>0$
we have that for any $t \leq r^2$ and $x,y \in X\cap\Beuc(x_0, \sqrt{t})$,
\[
q_t^A(x,y) \geq\frac{c}{m(\Beuc(x_0,\sqrt{t}))},
\]
where $A = \oX\cap\Beuc(x_0, C r)$.
\end{theorem}

\begin{pf}
This is obtained by combining Theorem~3.5 of \cite{Sturm} with
\eqref{strongpi} and Lemma~\ref{doubling} (giving parabolic Harnack
inequality), and then appealing to Theorem~3.2 in \cite{BGK} [the
assertion that (c) implies (b) is what we use with the function
$\tau(t)=t^2$]. Finally, using Proposition~\ref{P:bilipschitz} to move
from balls in $d_0$ to Euclidean balls.
\end{pf}

\section{Harmonic measure and exit time of discrete discs}
\label{sec-harmonic}

In this section, we prove Theorems \ref{T:exit_arg} and \ref{T:exit_time}. Let
$G=(V,E)$ be a
planar graph with a $(D,\eta)$-good embedding and let $(\oX,d_0)$ be the
associated metric space. We consider the cable process $Z$ on
$G$ defined in Section~\ref{sec-cableprocess}. We slightly abuse
notation, and
use $\tau_A$ to denote the hitting time of $A$ by the cable process,
that is,
$\tau_A = \inf\{t : Z_t\in A\}$. Recall that the restriction of $Z_t$ to
$V$ is the simple random walk, and so when $A\subset V$, the law of
$X_{\tau_A}$ is the same for the cable process and for the simple random
walk.

For $u\in\R^2$, radius $r>0$ and an interval of angles $I \subset
\R/(2\pi\Z)$ let\break $\cone(u,r,I)$ denote the intersection of $X$ and
the cone
of radius $r$ centered at $u$ with opening angles $I$, that is,
\[
\cone(u,r,I) = \bigl\{ v \in X : |v-u| \leq r \mbox{ and } \arg(v-u) \in I \bigr
\}.
\]
A \emph{wide cone} is a cone where $|I| \geq\pi-\eta$. By definition, if
$u$ is a vertex in a good embedding then there is an edge containing $u$
entering every wide cone with tip at~$u$.

\begin{lemma}\label{L:conevolume}
For any vertex $u\in V$ and any wide cone $A=\cone(u,r,I)$ such that
$\Beuc(u,r) \subset\carr(G)$, we have
\[
m(A) \asymp r(r \vee r_u).
\]
\end{lemma}

\begin{pf}
The case $r \leq r_u$ is easy and we omit the details. Assume $r \geq
r_u$ and write $A'$ for the Euclidean cone $A' = \{ x\in\R^2 :
|x-u|\leq
r \mbox{ and } \arg(x-u)\in I\}$. For any face $f$, we will show that
$m(A\cap\partial f) \geq c \leb(A'\cap f)$ for some $c(D,\eta)$. This
implies the lower bound, since summing over all faces gives $m(A) \geq
cr^2$.

Let $\ell= \diam(A'\cap f) \leq\diam(f)$, and note that every edge of
$f$ has length at least $c\ell$ for some $c$. Consider the circle $C_s =
\{z : |z-u|=s\}$. We have that the length $|C_s\cap f \cap A'|$ is at most
$C\ell$, and is nonzero for $s$ in some interval $J$. Integrating over
$s$ gives
\[
\leb\bigl(A'\cap f\bigr) = \int_J
\bigl|C_s\cap f \cap A'\bigr| \leq C \ell|J|.
\]

We next argue that $\partial f$ must cross inside $A$ any circle $C_s$
that intersects $f\cap A'$. To see this, note that we can construct a
path from
$u$ taking only edges with directions in $I$ until we exit $A$ after
finitely many steps [since $A\subset\carr(G)$]. The face $f$ is
restricted to one side of the path, and so $\partial f$ intersects
$A\cap
C_s$. It now follows that the length of $A\cap\partial f$ is at least
$|J|$, and since the length of edges of $f$ is at least $c\ell$ we get
$m(A\cap\partial f)\geq c\ell|J|$, and the lower bound follows.

For the upper bound, we prove only the case $r \geq r_u$, as the other is
immediate. By Lemma~\ref{L:sausage} every edge intersecting $A$ has
length at
most $Cr$, and all incident faces are contained in $\Beuc(u,C' r)$. For
any such edge $e$, taking all of $m(e)$ still gives at most the area of the
faces containing $e$, and since each face is counted a bounded number of
times, the claim follows by summing over the edges.
\end{pf}

\begin{corollary}\label{coneannulusvolume}
There exists a constant $c=c(D,\eta)>0$ such that for any vertex $u\in V$
and any $r\geq r_u$ with $\Beuc(u,r) \subset\carr(G)$, and any wide cone
$A=\cone(u,r,I)$ we have
\[
m \bigl(A \setminus\Beuc(u,cr) \bigr) \geq cr^2.
\]
\end{corollary}

\begin{lemma}\label{forwardincone}
There exist constants $c=c(D,\eta)>0$ such that for any interval
$I$ with $|I|=\pi-\eta$, any vertex $u\in V$ and any $r\geq r_u$ satisfying
$\Beuc(u,r) \subset\carr(G)$ we have
\[
\prob_u ( \tau_S < \tau_{V \setminus\Veuc(u,2r)} ) \geq c,
\]
where
\[
S = V \cap\cone(u,r,I) \setminus\Beuc(u,cr).
\]
\end{lemma}

\begin{pf}
Write $C$ for the unbounded Euclidean cone $\{x \in\R^2 : \arg(x-u)
\in
I\}$ and construct an infinite simple path $P$ from $u$ that remains in
$C$, as we did in the previous lemma. The existence of $P$ implies that
any edge $e$ that intersects $C$ must have at least one endpoint in $C$.

Write $c<1$ for the smaller of the constants in Corollaries \ref
{L:chorizo} and
\ref{coneannulusvolume}. Let $B \subset X \cap C$ be constructed as
follows: consider an edge $e$ that intersects $C \cap\Beuc(u,c r)
\setminus\Beuc(u,c^2 r)$ and does not contain $u$; if $e$ is entirely
contained in $C$, then we add $e$ to $B$, otherwise $e=(v_1, v_2)$ where
only $v_1$ is in $C$ and we add to $B$ the straight line segment between
$v_1$ and $(v_1+v_2)/2$ (i.e., half the edge $e$, starting at
$v_1$). We have that $m(B) \geq\frac{1}2 m  (X \cap C \cap\Beuc(u,c
r) \setminus\Beuc(u,c^2 r)  )$ since for any edge $e$ that
intersects $C \cap\Beuc(u, c r)$ we added to $B$ at least half of $e
\cap C$. Hence, by Corollary~\ref{coneannulusvolume} we get that $m(B)
\geq c^3
r^2/2$.

We now appeal to Theorem~\ref{T:heat} with $x_0=u$ and $t=r^2$ and integrate
over $y \in B$ to get that
\[
\prob_u ( Z_t \in B \mbox{ and } t < \tau_{\partial\Beuc(u,2r)} )
\geq c' > 0,
\]
for some constant $c'=c'(D,\eta)>0$.
By Corollary~\ref{L:chorizo}, we have that $B\cap V \subset X \cap C \cap
\Veuc(u,r)\setminus\Veuc(u,c^3 r)$ and since we added either full edges
or half edges, it is clear that starting from any point in $B$, the
probability that the first vertex that we visit is in $B \cap V$ is at
least $1/2$. This completes our proof.
\end{pf}

\begin{lemma}\label{exitconewithobstacles}
For any $\eps>0$, there exists $c=c(\eps, D, \eta)>0$ such that for any
vertex $u\in V$ and any $r\geq r_u$ satisfying
$\Beuc(u,r)\subset\carr(G)$, and any interval $I$ with $|I|=\pi
-\eta$ we
have
\[
\prob_u ( \tau_{S} < \tau_O ) > c,
\]
where
\[
S = V \cap\cone(u,\infty,I) \setminus\Veuc(u,r)
\]
and
\[
O = \bigl\{ v \in V : d\bigl(v, \cone(u,\infty,I)\bigr) \geq\eps r \bigr\}.
\]
\end{lemma}

\begin{pf}
We iterate $2(c\eps)^{-1}$ times Lemma~\ref{forwardincone} with a
cone of
radius $r'=\eps r/2$ and opening $I$.
\end{pf}

\begin{lemma}\label{goodarea}
For any $\eps>0$, there exists $c=c(\eps, D, \eta)>0$ such that for any
vertex $u\in V$, any $r\geq r_u$ satisfying $\Beuc(u,r)\subset\carr(G)$,
any interval $I$ with $|I|=\pi-\eta$, and any vertex $v$ such that
$\eps
r \leq|u-v| \leq(1-\eps)r$, and $\arg(v-u) \in I$ we have the
following. Let
\[
S = V \cap\cone(u,\infty,I) \setminus\Veuc(u,r)
\]
and
\[
Q = V \setminus\bigl(\cone(u,\infty,I) \cup\Veuc(u,r)\bigr),
\]
then $\prob_v ( \tau_S < \tau_Q) \geq c$.
\end{lemma}

\begin{pf}
Write $C$ for the Euclidean cone $\{x \in\R^2 : \arg(x-u) \in I\}$. If
the distance of $v$ from $\R^2 \setminus C$ is at least $\eps r$, then we
apply Lemma~\ref{exitconewithobstacles} on the cone parallel to $C$
emanating from $v$ and the assertion follows. Assume now the opposite,
and write $R_1, R_2$ for the two rays of the cone $C$ so that $R_1$ is
before $R_2$ clockwise and assume without loss of generality that $v$ is
closer to $R_1$. Let $C'$ be the cone
\[
C' = \bigl\{x \in\R^2 : \arg(x-v) \in I-\alpha\bigr\},
\]
where $\alpha= \alpha(\eps)>0$ is the largest number so that $d(u, C'
\setminus C) \geq(2c^{-1}+2) r$ where $c>0$ is the constant from
Corollary~\ref{L:chorizo} (see Figure~\ref{fig.goodarea}). Define the set
$O'$ by
\[
O' = \bigl\{ v \in V : d\bigl(v, V \cap C' \cap
\Veuc(v, 2r)\bigr) \geq\eps' r \bigr\},
\]
where $\eps'=\eps'(\eps,\alpha) >0$ is chosen so that $(V
\setminus O') \setminus\Veuc(u,r) \subset C$.

\begin{figure}[b]

\includegraphics{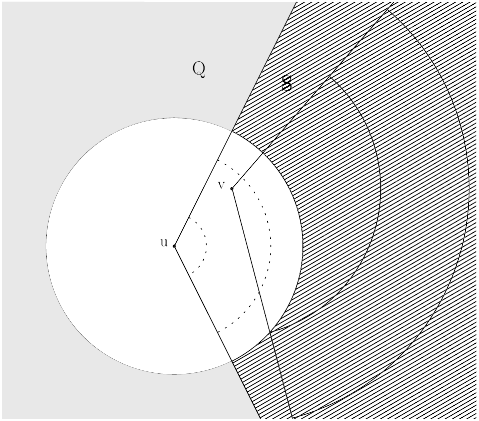}

\caption{Illustration of Lemma \protect\ref{goodarea}. The
probability from $v$ of
hitting $S$ before $Q$ cannot be too small. Also shown: the possible
locations for $v$, the rotated cone $C'$ and the set likely to be hit
from $v$ by Lemma \protect\ref{exitconewithobstacles}.}
\label{fig.goodarea}
\end{figure}

We now apply Lemma~\ref{exitconewithobstacles} with $\eps', v$ and
$C'$ to
obtain that with probability uniformly bounded below we visit $(V \cap
C') \setminus\Veuc(v,2r)$ before visiting $O'$. By Corollary~\ref{L:chorizo},
when this event occurs the length of the last edge traversed has length
at most $2c^{-1} r$. Hence, by our choice of $\alpha$ in this last
step we
find ourselves in $S$ and by our choice of $\eps'$ we have not stepped
outside of $C \cup\Veuc(u,r)$, as required.
\end{pf}

\begin{pf*}{Proof of Theorem~\ref{T:exit_arg}}
Let $\eps=\eps(D,\eta)>0$ be a fixed small number to be chosen
later. Let
us consider several cases. Recall that by condition (a) there always
exist an edge $(u,v)$ such that $\arg(v-u) \in I$. First, if there
exists such an edge $(u,v)$ with $|u-v| > r$, then with probability at
least $D^{-3}$ we take this edge in the first step and we are done.
Second, if there is such an edge so that $(1-\eps) r \leq|u-v| \leq
r$, then as long as $\eps$ is small with respect to $D$, then by
condition (b) $v$ has a neighbor $w$ such that $\arg(w-u) \in I$ and
$|w-u|>r$, so with probability at least $D^{-6}$ we take two steps from
$u$ to $w$ and we are done. Third, if there exists such an edge so
that $\eps r \leq|u-v| \leq(1-\eps)r$ then $\P_u(X_{T_r} \in S)
\geq
D^{-3} \P_v(X_{T_r}\in S)$ where $S$ is defined in Lemma~\ref{goodarea},
and by that lemma the last quantity is uniformly bounded below and we are
done. Lastly, if all neighbors $v$ of $u$ satisfy $|u-v|\leq\eps r$,
then we apply Lemma~\ref{exitconewithobstacles} with radius $\eps r$ and
obtain that with probability uniformly bounded from below we visit $V
\cap\cone(u,\eps r, I) \setminus\Veuc(u,\eps r)$ before visiting $O =
\{v \in V : d(v, \cone(u,\eps r, I ))\geq\eps^2 r\}$. When this occurs,
the last edge taken by the random walk has length at most $c^{-1} \eps r$
by Corollary~\ref{L:chorizo}, where $c>0$ is the constant of that lemma.
Hence, if
$\eps$ is chosen so that $\eps\leq(c^{-1}+2)^{-1}$ we get that at that
hitting time we are at a vertex $v$ such that $\eps r \leq|u-v| \leq
(1-\eps)r$ and $\arg(v-u)\in I$. The assertion of the theorem now follows
by another application of Lemma~\ref{goodarea}, completing the proof.
\end{pf*}

\begin{pf*}{Proof of Theorem~\ref{T:exit_time}}
Let $T_r^Z$ denote the exit time from $\Beuc(u,r)$ of the cable process
and $T_r$ the exit time from $\Veuc(u,r)$ for the simple random walk. Our
first goal is to prove that $\E_u T_r^Z \asymp r^2$. We begin with the
lower bound. To that aim, let $C_{\scriptsize{\ref{T:heat}}}$ be the
constant from Theorem~\ref{T:heat} and let
$A=\Beuc(u,C_{\scriptsize{\ref{T:heat}}} r)$ and assume that $A
\subset
\carr(G)$. We apply Theorem~\ref{T:heat} with $t=r^2$ and integrate
over $y\in
\Beuc(u, r)$ to get that $\prob(T_{C_{\scriptsize{\ref
{T:heat}}}r}^Z \geq
r^2) \geq c$, hence $\E_u T_{C_{\scriptsize{\ref{T:heat}}}r}^Z \geq cr^2$
and so $\E_u T_r^Z \geq c'r^2$ for some constant $c'>0$.

To show the upper bound, Lemma~\ref{doubling} immediately implies that there
exists some constant $C_{\scriptsize{\ref{doubling}}}>0$ such that for
any $r \geq r_u$ and any $x \in X \cap\Beuc(u,r)$ we have
\[
m \bigl(X \cap\Beuc(x,C_{\scriptsize{\ref{doubling}}}r) \setminus \Beuc(u,r) \bigr) \geq
r^2.
\]
We prove the theorem with $C = C_{\scriptsize{\ref{doubling}}}
C_{\scriptsize{\ref{T:heat}}} +1$. We apply Theorem~\ref{T:heat}
with $A =
\Beuc(x,C_{\scriptsize{\ref{doubling}}} C_{\scriptsize{\ref
{T:heat}}} r)$
[so that $A \subset\carr(G)$] and $t=r^2$ and integrate over $y \in X
\cap\Beuc(x,C_{\scriptsize{\ref{doubling}}}r) \setminus\Beuc
(u,r)$ to
get that for any $x \in X \cap\Beuc(u,r)$ we have $ \prob_x ( T_r
\geq
r^2 ) \leq1-c$, for some constant $c>0$. Hence, $\E_u T_r^Z \leq C' r^2$
for some $C'>0$.

We got that $\E_u(T_r^Z) \asymp r^2$. Recall that the
trace of the cable process along the vertices is distributed as the
discrete weighted random walk. Hence, by writing $T_r^Z$ as a the sum of
possible random walks paths and the time it takes the cable process to
traverse between vertices we obtain using \eqref{timetomove} that
\[
\E_u\bigl(T_r^Z\bigr) \asymp\E\sum
_{t=0}^{T_r} r^2_{X_t},
\]
where $\{X_t\}$ is the discrete weighted random walk.
\end{pf*}

\section{The Martin boundary} \label{sec-martin}

It will be convenient to approximate the graph $G$, embedded in the plane
with carrier $U$ by finite subgraphs $G_\eps$. For $\eps>0$, consider the
subgraph $G_\eps$ induced by the vertices $V_\eps$ where
\[
V_\eps= \bigl\{v \in V : |v| \leq1-\eps \bigr\}.
\]
For two vertices $a,z$ in a finite weighted graph we write $\Reff
(a,z)$ for
the effective electrical resistance between $a$ and $z$ (for a definition
and introduction to electrical resistance, see \cite{LP}). For disjoint
sets $A,Z$ of vertices, we write $\Reff(A, Z)$ for the electrical resistance
between $A$ and $Z$ in the graph obtained by contracting $A$ and $Z$ to two
vertices.

\begin{lemma}\label{L:resannulus}
There exists $c = c(D,\eta)>0$ such
that for any $r>0$ and $\eps\leq r/10$ and any $\xi\in\partial U$ we
have the resistance bound
\[
\Reff^{(\eps)} \bigl( \Veuc(\xi,r), \Veuc(\xi,2r)^c \bigr)
\geq c,
\]
where $\Reff^{(\eps)}$ denote the resistance is in the graph $G_\eps$.
\end{lemma}

\begin{pf} We use the discrete Dirichlet principle for effective
resistance; see Exercise 2.13 of \cite{LP}. Define a function
$f: V_\eps\to\R$ by
\[
f(x) = \cases{ 0, &\quad $\mbox{if } |x-\xi|\leq r$, \vspace*{2pt}
\cr
\displaystyle\frac{|x-\xi|-r}{r},
&\quad $\mbox{if } |x-\xi| \in[r,2r]$, \vspace*{2pt}
\cr
1, &\quad  $\mbox{if } |x-\xi| \geq2r
$,}
\]
and let us estimate the Dirichlet energy of the function. Note that $f$
is $r^{-1}$-Lipschitz, so that for any edge $(x,y)$ we have $|f(x) -
f(y)| \leq|x-y|/r$. All edges $(x,y)$ such that $x,y\in\Veuc(\xi,r)$
or $x,y \notin\Veuc(\xi, 2r)$ contribute $0$ to the energy. Any other
edge $(x,y)$ contributes at most $|x-y|^2/r^2$ to the energy. Since
$|x-y|^2$ is proportional to the area of the faces adjacent to the edge
$(x,y)$, all these faces are contained in $\Veuc(\xi, Cr)$ for some
$C=C(D,\eta)<\infty$, and each face has degree at most $C$, we get that
the energy is bounded by some constant and the result follows.
\end{pf}

\begin{corollary}\label{C:resannulus}
There exist constants $K=K(D,\eta) < \infty$ and $c = c(D, \eta)>0$ such
that for any $R,r$ satisfying $0 < Kr \leq R$, any $\eps\leq r/10$ and
any $\xi\in\partial U$ we have the resistance bound
\[
\Reff^{(\eps)} \bigl( \Veuc(\xi,r), V_\eps\setminus\Veuc(\xi,R)
\bigr) \geq c\log\frac{R}{r},
\]
where $\Reff^{(\eps)}$ denotes the resistance is in the graph $G_\eps$.
\end{corollary}

\begin{pf}
Let $K\geq2$ be such that there are no edges $(x,y)$ such that
$|x-\xi|\leq r$ and $|y-\xi|\geq Kr$. Such a choice is possible by
Lemma~\ref{L:sausage} and the fact that $\xi$ is an accumulation
point of
vertices.

Suppose first that $R=K^{2m-1}r$ for some integer $m\geq1$. Define sets
$A_0 = \{v : |v-\xi| < r\}$, and $A_i =  \{v : |v-\xi| \in[K^{i-1}
r, K^i r] \}$. By Lemma~\ref{L:sausage}, there are no edges connecting
$A_i$ to $A_j$ for $|i-j|>1$. By Lemma~\ref{L:resannulus}, we have that
$\Reff^{(\eps)}(A_i,A_{i+2}) \geq c$. Contracting all edges in $A_{2i}$
for each $i$ (recall that by Thompson's principle \cite{LP}, Chapter~2,
this operation can only decrease the effective resistance) and using the
series law for resistance we find
\[
\Reff^{(\eps)} \bigl( \Veuc(\xi,r), V_\eps\setminus\Veuc(\xi,R)
\bigr) \geq cm \geq c' \log\frac{R}{r}.
\]
For general $R>Kr$, the claim follows by monotonicity in $R$.
\end{pf}

\begin{proposition}[(Random walk convergence)]\label{P:rwconverge}
Let $X_n$ be the simple random walk on~$G$, then $X_n$ converges a.s. to
a limit $X_\infty\in\partial U$. Furthermore, the law of $X_\infty$
has no atoms.
\end{proposition}

Consequently, for any starting point $X_0$, we may define the harmonic
measure $\omega$ on $\partial U$ to be the law of $X_\infty$.

\begin{pf*}{Proof of Proposition \ref{P:rwconverge}}
By Lemma~\ref{L:face_diam}, for each edge $e=(u,v)$ we have that
$|v-u|^2$ is
bounded by a constant times the area of either faces adjacent to
$e$. Since each face has degree at most $2\pi/\eta$, this immediately
gives that the Dirichlet energy of the Euclidean location function,
that is,
$\sum_{e=(u,v)} |u-v|^2$, is bounded by some constant. By \cite{ALP}, Theorem~1.1, this implies that $X_n$ converges almost surely. (The theorem
is stated for real valued functions, so we apply it to each coordinate
separately.) It is trivial that the limit cannot be a vertex of $G$, so
must be in $\partial U$.

Let us now fix $X_0$, and show that for any $\xi\in\partial U$ we have
$\prob(X_\infty= \xi) = 0$. We have that $\Reff(X_0,\partial V_\eps)
\leq C$ for some $C$, since $G$ is transient. By Corollary~\ref{C:resannulus}
with $R=|X_0-\xi|$, for $r$ small enough and any $\eps<r/10$ we have
$\Reff^{(\eps)} (X_0,\Veuc(\xi,r) ) \geq c|\log r|$ and,
therefore,
\[
\P \bigl( \{X_n\} \mbox{ visits } \Veuc(\xi,r) \mbox{ before }
\partial V_\eps \bigr) \leq\frac{C}{|\log r|}.
\]
Since this estimate is uniform in $\eps$, we learn that the probability
that $\{X_n\}$ ever visits $\Veuc(\xi,r)$ is at most $C|\log
r|^{-1}$. This bound is uniform in $\eps$, and so holds also for the
random walk on $G$. Finally, $X_\infty= \xi$ implies that $\Veuc(\xi,r)$
is visited for all $r$, and so $\P(X_\infty= \xi) \leq\inf_r
C|\log
r|^{-1} = 0$.
\end{pf*}

We now state two variations of the Harnack principle that apply to well
embedded graphs.

\begin{theorem}[(Elliptic Harnack inequality)]\label{cableehi}
For any $A>1$, there exists $C=C(D,\eta,A)>0$ such that for any $x\in X$
and $r>0$ such that $d_0(x,\partial U) > Ar$, and any positive, harmonic
function $h$ on $B_{d_0}(x,Ar)$ we have
\[
\max_{y \in B_{d_0}(x,r)} h(y) \leq C \min_{y \in B_{d_0}(x,r)} h(y).
\]
\end{theorem}

\begin{theorem}[(Boundary Harnack principle)]\label{bhp}
There exists positive constants $A_0, A_1$ and $R$, depending only on $D$
and $\eta$, such that for any $\xi\in\partial U$, any $r \in(0,R)$ and
any two functions $h_1, h_2:X \to\R$ that are positive, harmonic,
bounded on $B_{d_0}(\xi, A_0 r)$, and almost surely $h_i(X_n) \to0$ as
$n\to\infty$ for $i=1,2$, we have
\[
A_1^{-1} \leq\frac{h_1 (x) / h_2(x)}{h_1(y) / h_2(y)} \leq A_1
\qquad\forall x,y \in B_{d_0}(\xi, r) \cap X.
\]
\end{theorem}

Theorem~\ref{cableehi} follows from Theorem~3.5 of \cite{Sturm}.
To obtain Theorem~\ref{bhp}, we use~\cite{LS}, Theorem~4.2. We take
their $\hat\sE$ and $\sE$ to be our $\sE$, their spaces
$X$ and $Y$ to be $\oX$, and their $\Omega$ to be $X$.
Since in this case $\hat\sE= \sE$ Assumptions 1 and 2 of \cite{LS}
hold, and the conditions of volume doubling and the Poincar\'e inequality
needed in \cite{LS} are provided by Theorem~\ref{T:weakpi} and Lemma \ref{doubling}.
Finally, Lemma~\ref{L:inneruniform} shows that $\Omega$ is inner uniform.

We now proceed to the proof of Theorem~\ref{T:martinbdry}. It has been known
since Ancona \cite{Ancona} that a boundary Harnack principle such as
Theorem~\ref{bhp} implies that the Martin boundary is homeomorphic to the
Euclidean boundary. The papers \cite{A,ALM,LS} all contain results of
this kind. In particular, the proof in \cite{A} is quite robust, and
translates to our setting with only minor changes. However, since the
argument is both reasonably short and illuminating, we include it for the
sake of completeness.

%

For convenience, we consider the Martin kernels as a function of the first
coordinate, that is, we denote $M_y(\cdot) = G(\cdot,y)/G(x_0,y)$. Let
$\cH^+$ denote the set of positive harmonic functions $h$ on $X$,
normalized to have $h(x_0)=1$. Note that on any locally finite connected
graph, $\cH^+$ is compact w.r.t. the product (pointwise) topology. Of
those, we let $\cH^+_0$ denote the set of functions so that
$h(X_n)\mathop{\rightarrow}\limits_{n\to\infty}^{\mathrm{a.s.}} 0$ for any starting point
$X_0$. By
the martingale convergence theorem \cite{D}, a.s. convergence holds for
any positive harmonic function; it is of course enough to assume the limit
is a.s. $0$ for a single starting point. Finally, for $\xi\in
\partial U$
let us denote by $\cH_\xi$ those functions $h\in\cH^+_0$ which are bounded
on $X \setminus B_{d_0}(\xi,r)$ for any $r>0$. Our immediate goal is the
following.

\begin{proposition}\label{P:uniqueness}
For any $\xi\in\partial U$, the set $\cH_\xi$ is a singleton.
\end{proposition}

We first prove that $\cH_\xi$ is not empty.

\begin{lemma}\label{L:martinexists}
Let $y_n$ be a sequence of vertices and suppose $y_n \to\xi\in
\partial
U$. Then there exists a subsequence $y_{n_k}$ such that $M_{y_{n_k}}$
converges pointwise to some $h\in\cH_\xi$.
\end{lemma}

\begin{pf}
Since $\cH^+$ is compact, there exist a subsequence $y_{n_k}$ such that
$M_{y_{n_k}}$ converges pointwise. For clarity, we pass to the
subsequence. Let $M_\xi$ be the limit. Let us now prove that $M_\xi
\in
H_\xi$. It is clear that $M_\xi$ is harmonic and $M_\xi(x_0)=1$ since
these are local constraints and are immediately satisfied by the limiting
procedure, so we need to show that
$M_\xi(X_n)\mathop{\rightarrow}\limits_{n\to\infty}^{\mathrm{a.s.}} 0$ and that $M_\xi$ is bounded
outside any neighborhood of $\xi$.

Recall that by the reversibility of the random walk we have
%
\begin{equation}
\label{reversible} \deg(x) \cdot G(x,y) = \deg(y) \cdot G(y,x),
\end{equation}
and since degrees are bounded, $G(x,y)$ and $G(y,x)$ are equivalent up to
constants. We therefore have that
\[
M_{y_k}(x) \asymp\frac{G(y_k,x)}{G(y_k,x_0)}.
\]

Let $A_0$ be the constant from Theorem~\ref{bhp}, let $r>0$ be
arbitrary small
such that $x_0 \notin B_{d_0}(\xi,A_0 r)$ and let $x$ be an arbitrary
vertex satisfying $x \notin B_{d_0} (\xi, A_0 r)$. Define the functions
$h_0=G(\cdot,x_0)$ and $h_1=G(\cdot,x)$. The functions $h_0, h_1$ are
positive, harmonic on $X \cap B_{d_0}(\xi,r)$ and bounded above by
$G(x,x)$ and $G(x_0,x_0)$, respectively. Furthermore, both tend to $0$
almost surely over the random walk since $G$ is a transient graph. Hence,
we may apply Theorem~\ref{bhp} to them and deduce that
\[
{G(z, x) \over G(z, x_0)} \asymp{G(v_r, x) \over G(v_r, x_0)},
\]
where $v_r,z$ are any two vertices in $B_{d_0}(\xi,r)$ and the constants
in the $\asymp$ do not depend on the choice of $x$. Let $k_0$ be a
number so that for all $k \geq k_0$ we have $y_{k}\in B_{d_0}(\xi,r)$ so
by the previous approximate equality we get that for any $k\geq k_0$ we
have
%
\begin{equation}
\label{boundusebhp} M_{y_k}(x) \asymp\frac{G(v_r, x)}{G(v_r,x_0)} \asymp
\frac{G(x,v_r)}{G(x_0,v_r)},
\end{equation}
for all $x \notin B_{d_0} (\xi, A_0 r)$ and $v_r$ a fixed vertex (its
choice may depend on $r$). Since $G(x,v_r) \leq G(v_r,v_r)$, we learn that
$M_{y_k}$ is bounded outside of $B_{d_0} (\xi, r)$ for any $r>0$ and
$k>k_0$, and we deduce the same for $M_\xi$ immediately.

Next, by Proposition~\ref{P:rwconverge} the probability that $\lim X_t
= \xi$ is
$0$. We learn that almost surely there exists $r>0$ such that $X_t
\notin B_{d_0}(\xi, A_0 r)$ for all $t \geq0$. Let $k_0$ be as above. By
\eqref{boundusebhp}, we get that almost surely for any $t \geq0$
\[
M_{y_{k}}(X_t) \asymp\frac{G(v_r, X_t)}{G(v_r, x_0)},
\]
and by taking a limit $k \to\infty$ we have that almost surely
\[
M_{\xi}(X_t) \leq A {G(v_r, X_t) \over G(v_r, x_0)},
\]
for all $t\geq0$ where $A = A(D,\eta)<\infty$. Since $G(v_r,X_t) \to0$
as $t \to\infty$ almost surely, we deduce that $\lim M_\xi(X_n) = 0$
almost surely, concluding the proof.
\end{pf}

\begin{pf*}{Proof of Proposition~\ref{P:uniqueness}}
We first prove that there exists $A=A(D,\eta)<\infty$ such that for any
$h_1,h_2 \in\cH_\xi$ we have
%
\begin{equation}
\label{martinstep1} A^{-1} \leq\frac{h_1(x)}{h_2(x)} \leq A \qquad\mbox{for
all $x\in X$.}
\end{equation}
Let $r>0$ be an arbitrary small number and let $\xi_1, \xi_2 \in
\partial
U$ be the two boundary points so that $|\xi-\xi_1|=|\xi-\xi_2|=r$.
We appeal
to Lemma~\ref{L:inneruniform} and get a curve $\Gamma:(0,L) \to X$
satisfying the conditions of the lemma. We now use the curve to construct
balls $B_0, \ldots, B_N$ for some $N=N(D,\eta)<\infty$ such that for some
small $c\in(0,1/2)$ the following holds:
\begin{longlist}[(1)]
\item[(1)]$B_0 = B_{d_0}(\xi_1, r/(2A_0))$ and $B_N =
B_{d_0}(\xi_2,r/(2A_0))$,
\item[(2)] for $i=1,\ldots, N-1$ we have $B_i = B_{d_0}(x_i, cr)$ where
$x_i \in\gamma$ and $d_0(x_i,\partial U) > 2 c r$,
\item[(3)]$B_i \cap B_{i+1} \neq\varnothing$ for $i=0,\ldots, N-1$.
\end{longlist}
We apply Theorems \ref{bhp} and \ref{cableehi} to obtain that there exists
$A=A(D,\eta) <
\infty$ such that
\[
A^{-1} \leq\frac{h_1(x)/h_2(x)}{h_1(x')/h_2(x')} \leq A \qquad \forall x,x'
\in\bigcup_{i=1}^N B_i.
\]
Indeed, the assertion for $x,x'\in B_0$ and $x, x'\cup B_N$ is precisely
Theorem~\ref{bhp}. Moreover, Theorem~\ref{cableehi} gives that the
values of $h_1,h_2$
within $B_1 \cup\cdots\cup B_{N-1}$ change by at most a multiplicative
constant.

Fix $x' \in\gamma$ and note that $h_1(x)/h_2(x) \leq q$ for all $x\in
\gamma$ where $q = A h_1(x')/ h_2(x')$. Then $g(x) = h_1(x) - q h_2(x)$ is
harmonic and nonpositive on $\gamma$. The martingale $g(X_n)$ stopped
when hitting $\gamma$ is bounded, converges to $0$ if $\gamma$ is not
hit, and is stopped at a negative value if $\gamma$ is hit. By
$L^1$-convergence for bounded martingales $g(x)\leq0$ everywhere, and so
\[
A^{-1} \leq\frac{h_1(x)/h_2(x)}{h_1(x')/h_2(x')} \leq A \qquad \forall x \in X \setminus
B_{d_0}(\xi,r).
\]
In particular, $h_1(x')/h_2(x') \leq A h_1(x_0)/h_2(x_0) = A$ and
similarly $h_1(x')/\break  h_2(x') \geq A^{-1}$. Hence, $A^{-2} \leq h_1(x)/h_2(x)
\leq A^2$ for all $x \in X\setminus\Veuc(\xi,r)$. Since $r>0$ was
arbitrary, this gives \eqref{martinstep1}.

Next, we show that in fact $A=1$; the following argument is due to Ancona
\cite{Ancona}. Indeed, write
\[
c = \sup_{h_1,h_2 \in H_\xi, x\in X} {h_1(x) \over h_2(x)},
\]
so that $c\in[1,\infty)$. Assume by contradiction that $c>1$ and let
$h_1,h_2\in H_\xi$. Then $h_3 = (ch_1 - h_2)/(c-1)$ is a function in
$H_\xi$ so $h_2 \leq ch_3$ which simplifies to $(2c-1)h_2 \leq c^2
h_1$. Since $c^2/(2c-1) < c$, we have reached a contradiction.
\end{pf*}

\begin{pf*}{Proof of Theorem~\ref{T:martinbdry}$'$}
Minimality of $M_\xi$ follows easily from Proposition~\ref
{P:uniqueness}, since if
$0\leq h\leq M_\xi$ then $h(\cdot)/h(x_0)$ is easily seen to be in
$\cH_\xi$, and so it must equal $M_\xi$.

Suppose $y_n\to\xi\in\partial U$ then Proposition~\ref
{P:uniqueness} and Lemma~\ref{L:martinexists}
together show that $\lim_{y_n \to\xi} M_{y_n}(\cdot)$ exists and is the
unique function in $\cH_\xi$. Thus, convergence of $y_n$ implies
convergence of $M_{y_n}$.

Next, note that if $\xi\neq\xi'$ are two points on $\partial U$, then
$M_{\xi}\neq M_{\xi'}$. Indeed, $M_\xi$ is an unbounded function, since
otherwise, by the bounded martingale convergence theorem we would get
that $\E\lim M_\xi(X_n) = M_\xi(x_0)=1$, contradicting the fact that
$M_\xi(X_n) \to0$ almost surely. However, $M_\xi$ is bounded away from
$\xi$ and so must be unbounded in any neighborhood of $\xi$. It follows
that $M_\xi\neq M_{\xi'}$.

Now, suppose we have a convergent sequence $M_{y_n}\to M_\infty$ for some
sequence $y_n$. Since $U$ is compact, there is a convergent subsequence
$y_{n_k}\to\xi$. If $\xi$ is not in $\partial U$ then eventually
$y_{n_k}=\xi$. Otherwise, $M_{y_{n_k}}\to M_\xi$, and in either case
$M_\infty=M_\xi$. Since $\xi$ is determined by $M_\xi$, we have that
$y_n\to\xi$, completing the proof of~(1).

Finally, we show that the map $\xi\mapsto M_\xi(\cdot)$ is a
homeomorphism. It is invertible, so we need continuity of the map and its
inverse. Suppose $\xi_n\to\xi$ are points in $\partial U$. For an
arbitrary $x$, we may find $y_n$ so that $d(y_n,\xi_n)<\frac{1}n$, and also
$|M_{y_n}(x)-M_{\xi_n}(x)|\leq\frac{1}n$. We have that $y_n\to\xi$ and,
therefore, $M_{y_n}(x)\to M_\xi(x)$, and so also $M_{\xi_n}(x)\to
M_\xi(x)$.
Similarly, if $M_{\xi_n} \to M_\xi$ we can diagonalize to find $y_n$ with
$d(y_n,\xi_n)<\frac{1}n$ so that $M_{y_n}\to M_\xi$. By (1), we have
$y_n\to\xi$ and, therefore, $\xi_n\to\xi$.
\end{pf*}

\begin{pf*}{Proof of Theorem~\ref{T:mainmainthm}$'$}
We appeal to general properties of the Martin boundary; see Chapter~24 of
\cite{woess} for a concise introduction. This theory implies that any
positive harmonic function $h$ can be represented as an integral on the
Martin boundary $\mathcal{M}$ with respect to some measure. When $h$ is
bounded, this measure is absolutely continuous with respect to the exit
measure on the Martin boundary, hence it can be written as
%
\begin{equation}
\label{w1} h(x) = \int_\mathcal{M} M(x) f(M) \,d
\nu_{x_0}(M),
\end{equation}
where $\nu_{x_0}$
is the law of $\lim_n M_{X_n}(\cdot)$ starting from $x_0$ and
$f:\mathcal{M} \to\R$ is some bounded measurable function; see Theorem~24.12 in \cite{woess}. Theorem~24.10 in \cite{woess} states that the
Radon--Nikodym derivative of $\nu_x$ with respect to $\nu_{x_0}$ is the
function from $\mathcal{M}$ to $\R$ mapping each $M\in\mathcal{M}$ to
$M(x)$. Hence, we may rewrite \eqref{w1} as
\[
h(x) = \int_\mathcal{M} f(M) \,d \nu_x(M).
\]
Now, apply Theorem~\ref{T:martinbdry} and let $\iota: \partial U \to
\mathcal{M}$ be the homeomorphism $\xi\mapsto
M_\xi$. Theorem~\ref{T:martinbdry} implies that the image under
$\iota$ of the
random walk's exit measure on $\partial U$ coincides with $\nu_x$,
completing our proof.
\end{pf*}

\section*{Acknowledgement}
We are grateful to Dmitry Chelkak for useful discussions.

%

\printaddresses

\begin{thebibliography}{31}

\bibitem{A}
%
\begin{barticle}[mr]
\bauthor{\bsnm{Aikawa},~\bfnm{Hiroaki}\binits{H.}}
(\byear{2001}).
\btitle{Boundary {H}arnack principle and {M}artin boundary for a
uniform domain}.
\bjournal{J. Math. Soc. Japan}
\bvolume{53}
\bpages{119--145}.
\bid{doi={10.2969/jmsj/05310119}, issn={0025-5645}, mr={1800526}}
\end{barticle}
%
\bptok{imsref}%
\endbibitem

\bibitem{ALM}
%
\begin{barticle}[mr]
\bauthor{\bsnm{Aikawa},~\bfnm{Hiroaki}\binits{H.}},
\bauthor{\bsnm{Lundh},~\bfnm{Torbj{\"o}rn}\binits{T.}} \AND
\bauthor{\bsnm{Mizutani},~\bfnm{Tomohiko}\binits{T.}}
(\byear{2003}).
\btitle{Martin boundary of a fractal domain}.
\bjournal{Potential Anal.}
\bvolume{18}
\bpages{311--357}.
\bid{doi={10.1023/A:1021823023212}, issn={0926-2601}, mr={1953266}}
\end{barticle}
%
\bptok{imsref}%
\endbibitem

\bibitem{Ancona}
%
\begin{barticle}[mr]
\bauthor{\bsnm{Ancona},~\bfnm{Alano}\binits{A.}}
(\byear{1978}).
\btitle{Principe de {H}arnack \`a la fronti\`ere et th\'eor\`eme de
{F}atou pour un op\'erateur elliptique dans un domaine lipschitzien}.
\bjournal{Ann. Inst. Fourier (Grenoble)}
\bvolume{28}
\bpages{169--213}.
\bid{issn={0373-0956}, mr={0513885}}
\bptnote{check pages}%
\end{barticle}
%
\bptok{imsref}%
\endbibitem

\bibitem{ALP}
%
\begin{barticle}[mr]
\bauthor{\bsnm{Ancona},~\bfnm{Alano}\binits{A.}},
\bauthor{\bsnm{Lyons},~\bfnm{Russell}\binits{R.}} \AND
\bauthor{\bsnm{Peres},~\bfnm{Yuval}\binits{Y.}}
(\byear{1999}).
\btitle{Crossing estimates and convergence of {D}irichlet functions
along random walk and diffusion paths}.
\bjournal{Ann. Probab.}
\bvolume{27}
\bpages{970--989}.
\bid{doi={10.1214/aop/1022677392}, issn={0091-1798}, mr={1698991}}
\end{barticle}
%
\bptok{imsref}%
\endbibitem

\bibitem{BGK}
%
\begin{barticle}[mr]
\bauthor{\bsnm{Barlow},~\bfnm{Martin~T.}\binits{M.~T.}},
\bauthor{\bsnm{Grigor'yan},~\bfnm{Alexander}\binits{A.}} \AND
\bauthor{\bsnm{Kumagai},~\bfnm{Takashi}\binits{T.}}
(\byear{2012}).
\btitle{On the equivalence of parabolic {H}arnack inequalities and
heat kernel estimates}.
\bjournal{J. Math. Soc. Japan}
\bvolume{64}
\bpages{1091--1146}.
\bid{doi={10.2969/jmsj/06441091}, issn={0025-5645}, mr={2998918}}
\end{barticle}
%
\bptok{imsref}%
\endbibitem

\bibitem{BB}
%
\begin{barticle}[mr]
\bauthor{\bsnm{Bass},~\bfnm{Richard~F.}\binits{R.~F.}} \AND
\bauthor{\bsnm{Burdzy},~\bfnm{Krzysztof}\binits{K.}}
(\byear{1991}).
\btitle{A boundary {H}arnack principle in twisted H\"older domains}.
\bjournal{Ann. of Math. (2)}
\bvolume{134}
\bpages{253--276}.
\bid{doi={10.2307/2944347}, issn={0003-486X}, mr={1127476}}
\end{barticle}
%
\bptok{imsref}%
\endbibitem

\bibitem{BS2}
%
\begin{barticle}[mr]
\bauthor{\bsnm{Benjamini},~\bfnm{Itai}\binits{I.}} \AND
\bauthor{\bsnm{Schramm},~\bfnm{Oded}\binits{O.}}
(\byear{1996}).
\btitle{Harmonic functions on planar and almost planar graphs and
manifolds, via circle packings}.
\bjournal{Invent. Math.}
\bvolume{126}
\bpages{565--587}.
\bid{doi={10.1007/s002220050109}, issn={0020-9910}, mr={1419007}}
\end{barticle}
%
\bptok{imsref}%
\endbibitem

\bibitem{Chelkak}
%
\begin{barticle}[auto:parserefs-M02]
\bauthor{\bsnm{Chelkak},~\bfnm{Dmitry}\binits{D.}}
(\byear{2016}).
\btitle{Robust discrete complex analysis: A toolbox}.
\bjournal{Ann. Probab.}
\bvolume{44}
\bpages{628--683}.
\bid{mr={3456348}}
\end{barticle}
%
\bptok{imsref}%
\endbibitem

\bibitem{CF}
%
\begin{bbook}[mr]
\bauthor{\bsnm{Chen},~\bfnm{Zhen-Qing}\binits{Z.-Q.}} \AND
\bauthor{\bsnm{Fukushima},~\bfnm{Masatoshi}\binits{M.}}
(\byear{2012}).
\btitle{Symmetric {M}arkov Processes, Time Change, and Boundary Theory}.
\bseries{London Mathematical Society Monographs Series}
\bvolume{35}.
\bpublisher{Princeton Univ. Press},
\blocation{Princeton, NJ}.
\bid{mr={2849840}}
\end{bbook}
%
\bptok{imsref}%
\endbibitem

\bibitem{De}
%
\begin{barticle}[mr]
\bauthor{\bsnm{Derriennic},~\bfnm{Yves}\binits{Y.}}
(\byear{1976}).
\btitle{Lois ``z\'ero ou deux'' pour les processus de {M}arkov.
{A}pplications aux marches al\'eatoires}.
\bjournal{Ann. Inst. H. Poincar\'e Sect. B (N.S.)}
\bvolume{12}
\bpages{111--129}.
\bid{mr={0423532}}
\end{barticle}
%
\bptok{imsref}%
\endbibitem

\bibitem{D}
%
\begin{bbook}[mr]
\bauthor{\bsnm{Durrett},~\bfnm{Rick}\binits{R.}}
(\byear{2010}).
\btitle{Probability: Theory and Examples},
\bedition{4th} ed.
\bpublisher{Cambridge Univ. Press},
\blocation{Cambridge}.
\bid{doi={10.1017/CBO9780511779398}, mr={2722836}}
\end{bbook}
%
\bptok{imsref}%
\endbibitem

\bibitem{Dynkin}
%
\begin{barticle}[mr]
\bauthor{\bsnm{Dynkin},~\bfnm{E.~B.}\binits{E.~B.}}
(\byear{1969}).
\btitle{The boundary theory of {M}arkov processes (discrete case)}.
\bjournal{Uspekhi Mat. Nauk}
\bvolume{24}
\bpages{3--42}.
\bid{issn={0042-1316}, mr={0245096}}
\end{barticle}
%
\bptok{imsref}%
\endbibitem

\bibitem{Folz}
%
\begin{barticle}[mr]
\bauthor{\bsnm{Folz},~\bfnm{Matthew}\binits{M.}}
(\byear{2014}).
\btitle{Volume growth and stochastic completeness of graphs}.
\bjournal{Trans. Amer. Math. Soc.}
\bvolume{366}
\bpages{2089--2119}.
\bid{doi={10.1090/S0002-9947-2013-05930-2}, issn={0002-9947}, mr={3152724}}
\end{barticle}
%
\bptok{imsref}%
\endbibitem

\bibitem{fot}
%
\begin{bbook}[mr]
\bauthor{\bsnm{Fukushima},~\bfnm{Masatoshi}\binits{M.}},
\bauthor{\bsnm{{\=O}shima},~\bfnm{Y{\=o}ichi}\binits{Y.}} \AND
\bauthor{\bsnm{Takeda},~\bfnm{Masayoshi}\binits{M.}}
(\byear{1994}).
\btitle{Dirichlet Forms and Symmetric {M}arkov Processes}.
\bseries{De Gruyter Studies in Mathematics}
\bvolume{19}.
\bpublisher{de Gruyter},
\blocation{Berlin}.
\bid{doi={10.1515/9783110889741}, mr={1303354}}
\end{bbook}
%
\bptok{imsref}%
\endbibitem

\bibitem{Agelos}
%
\begin{bmisc}[auto:parserefs-M02]
\bauthor{\bsnm{Georgakopoulos},~\bfnm{Agelos}\binits{A.}}
(\byear{2016}).
\bhowpublished{The boundary of a square tiling of a graph coincides
with the Poisson boundary. \textit{Inventiones Mathematicae}. To appear}.
\end{bmisc}
%
\bptok{imsref}%
\endbibitem

\bibitem{HS}
%
\begin{barticle}[mr]
\bauthor{\bsnm{He},~\bfnm{Zheng-Xu}\binits{Z.-X.}} \AND
\bauthor{\bsnm{Schramm},~\bfnm{O.}\binits{O.}}
(\byear{1995}).
\btitle{Hyperbolic and parabolic packings}.
\bjournal{Discrete Comput. Geom.}
\bvolume{14}
\bpages{123--149}.
\bid{doi={10.1007/BF02570699}, issn={0179-5376}, mr={1331923}}
\end{barticle}
%
\bptok{imsref}%
\endbibitem

\bibitem{Jerison}
%
\begin{barticle}[mr]
\bauthor{\bsnm{Jerison},~\bfnm{David}\binits{D.}}
(\byear{1986}).
\btitle{The {P}oincar\'e inequality for vector fields satisfying H\"
ormander's condition}.
\bjournal{Duke Math. J.}
\bvolume{53}
\bpages{503--523}.
\bid{doi={10.1215/S0012-7094-86-05329-9}, issn={0012-7094}, mr={0850547}}
\end{barticle}
%
\bptok{imsref}%
\endbibitem

\bibitem{JK}
%
\begin{barticle}[mr]
\bauthor{\bsnm{Jerison},~\bfnm{David~S.}\binits{D.~S.}} \AND
\bauthor{\bsnm{Kenig},~\bfnm{Carlos~E.}\binits{C.~E.}}
(\byear{1982}).
\btitle{Boundary behavior of harmonic functions in nontangentially
accessible domains}.
\bjournal{Adv. in Math.}
\bvolume{46}
\bpages{80--147}.
\bid{doi={10.1016/0001-8708(82)90055-X}, issn={0001-8708}, mr={0676988}}
\end{barticle}
%
\bptok{imsref}%
\endbibitem

\bibitem{Ka}
%
\begin{bincollection}[mr]
\bauthor{\bsnm{Kaimanovich},~\bfnm{Vadim~A.}\binits{V.~A.}}
(\byear{1992}).
\btitle{Measure-theoretic boundaries of {M}arkov chains, {$0$}--{$2$}
laws and entropy}.
In \bbooktitle{Harmonic Analysis and Discrete Potential Theory
({F}rascati, 1991)}
\bpages{145--180}.
\bpublisher{Plenum},
\blocation{New York}.
\bid{mr={1222456}}
\end{bincollection}
%
\bptok{imsref}%
\endbibitem

\bibitem{Ko}
%
\begin{barticle}[auto:parserefs-M02]
\bauthor{\bsnm{Koebe},~\bfnm{P.}\binits{P.}}
(\byear{1936}).
\btitle{Kontaktprobleme der Konformen Abbildung}.
\bjournal{Ber. S\"{a}chs. Akad. Wiss. Leipzig, Math.-Phys. Kl.}
\bvolume{88}
\bpages{141--164}.
\end{barticle}
%
\bptok{imsref}%
\endbibitem

\bibitem{LS}
%
\begin{barticle}[mr]
\bauthor{\bsnm{Lierl},~\bfnm{Janna}\binits{J.}} \AND
\bauthor{\bsnm{Saloff-Coste},~\bfnm{Laurent}\binits{L.}}
(\byear{2014}).
\btitle{Scale-invariant boundary {H}arnack principle in inner uniform domains}.
\bjournal{Osaka J. Math.}
\bvolume{51}
\bpages{619--656}.
\bid{issn={0030-6126}, mr={3272609}}
\bptnote{check volume, check pages, check year}%
\end{barticle}
%
\bptok{imsref}%
\endbibitem

\bibitem{LP}
%
\begin{bmisc}[author]
\bauthor{\bsnm{Lyons},~\bfnm{Russel}\binits{R.}} \AND
\bauthor{\bsnm{Peres},~\bfnm{Yuval}\binits{Y.}}
(\byear{2014}).
\bhowpublished{Probability on trees and networks.
Cambridge Univ. Press,
Preprint. Available at
\surl{http://mypage.iu.edu/\textasciitilde rdlyons/\\prbtree/book.pdf}}.
\end{bmisc}
%
\bptok{imsref}%
\endbibitem

\bibitem{Martin}
%
\begin{barticle}[mr]
\bauthor{\bsnm{Martin},~\bfnm{Robert~S.}\binits{R.~S.}}
(\byear{1941}).
\btitle{Minimal positive harmonic functions}.
\bjournal{Trans. Amer. Math. Soc.}
\bvolume{49}
\bpages{137--172}.
\bid{issn={0002-9947}, mr={0003919}}
\end{barticle}
%
\bptok{imsref}%
\endbibitem

\bibitem{RS}
%
\begin{barticle}[mr]
\bauthor{\bsnm{Rodin},~\bfnm{Burt}\binits{B.}} \AND
\bauthor{\bsnm{Sullivan},~\bfnm{Dennis}\binits{D.}}
(\byear{1987}).
\btitle{The convergence of circle packings to the {R}iemann mapping}.
\bjournal{J. Differential Geom.}
\bvolume{26}
\bpages{349--360}.
\bid{issn={0022-040X}, mr={0906396}}
\end{barticle}
%
\bptok{imsref}%
\endbibitem

\bibitem{R}
%
\begin{barticle}[mr]
\bauthor{\bsnm{Rohde},~\bfnm{Steffen}\binits{S.}}
(\byear{2011}).
\btitle{Oded {S}chramm: From circle packing to {SLE}}.
\bjournal{Ann. Probab.}
\bvolume{39}
\bpages{1621--1667}.
\bid{doi={10.1214/10-AOP590}, issn={0091-1798}, mr={2884870}}
\end{barticle}
%
\bptok{imsref}%
\endbibitem

\bibitem{SC}
%
\begin{bbook}[mr]
\bauthor{\bsnm{Saloff-Coste},~\bfnm{Laurent}\binits{L.}}
(\byear{2002}).
\btitle{Aspects of {S}obolev-Type Inequalities}.
\bseries{London Mathematical Society Lecture Note Series}
\bvolume{289}.
\bpublisher{Cambridge Univ. Press},
\blocation{Cambridge}.
\bid{mr={1872526}}
\end{bbook}
%
\bptok{imsref}%
\endbibitem

\bibitem{St}
%
\begin{bbook}[mr]
\bauthor{\bsnm{Stephenson},~\bfnm{Kenneth}\binits{K.}}
(\byear{2005}).
\btitle{Introduction to Circle Packing: The Theory of Discrete
Analytic Functions}.
\bpublisher{Cambridge Univ. Press},
\blocation{Cambridge}.
\bid{mr={2131318}}
\end{bbook}
%
\bptok{imsref}%
\endbibitem

\bibitem{Sturm}
%
\begin{barticle}[mr]
\bauthor{\bsnm{Sturm},~\bfnm{K.~T.}\binits{K.~T.}}
(\byear{1996}).
\btitle{Analysis on local {D}irichlet spaces. III. {T}he parabolic
{H}arnack inequality}.
\bjournal{J. Math. Pures Appl. (9)}
\bvolume{75}
\bpages{273--297}.
\bid{issn={0021-7824}, mr={1387522}}
\end{barticle}
%
\bptok{imsref}%
\endbibitem

\bibitem{Var}
%
\begin{barticle}[mr]
\bauthor{\bsnm{Varopoulos},~\bfnm{Nicholas~T.}\binits{N.~T.}}
(\byear{1985}).
\btitle{Long range estimates for {M}arkov chains}.
\bjournal{Bull. Math. Sci.}
\bvolume{109}
\bpages{225--252}.
\bid{issn={0007-4497}, mr={0822826}}
\end{barticle}
%
\bptok{imsref}%
\endbibitem

\bibitem{Wal}
%
\begin{barticle}[auto:parserefs-M02]
\bauthor{\bsnm{Walsh},~\bfnm{J.~B.}\binits{J.~B.}}
(\byear{1978}).
\btitle{A diffusion with a discontinuous local time}.
\bjournal{Temps Locaux, Ast\'erisque}
\bvolume{52--53}
\bpages{37--45}.
\end{barticle}
%
\bptok{imsref}%
\endbibitem

\bibitem{woess}
%
\begin{bbook}[mr]
\bauthor{\bsnm{Woess},~\bfnm{Wolfgang}\binits{W.}}
(\byear{2000}).
\btitle{Random Walks on Infinite Graphs and Groups}.
\bseries{Cambridge Tracts in Mathematics}
\bvolume{138}.
\bpublisher{Cambridge Univ. Press},
\blocation{Cambridge}.
\bid{doi={10.1017/CBO9780511470967}, mr={1743100}}
\end{bbook}
%
\bptok{imsref}%
\endbibitem

\end{thebibliography}
\end{document}